\documentclass[11pt]{amsart}

\usepackage{cite}
\usepackage{amsmath,amssymb,amsthm,mathtools,bm}
\usepackage{graphicx}
\usepackage{booktabs}
\usepackage{algorithm}
\usepackage{algpseudocode}
\usepackage{flafter}
\usepackage{placeins}
\usepackage{url}
\usepackage[hidelinks]{hyperref}
\usepackage[nameinlink,capitalize,noabbrev]{cleveref}
\usepackage{microtype}

\newtheorem{theorem}{Theorem}
\newtheorem{lemma}{Lemma}
\newtheorem{proposition}{Proposition}
\newtheorem{corollary}{Corollary}
\newtheorem{assumption}{Assumption}
\theoremstyle{definition}
\newtheorem{definition}{Definition}
\newtheorem{remark}{Remark}

\crefname{assumption}{Assumption}{Assumptions}
\Crefname{assumption}{Assumption}{Assumptions}

\newcommand{\R}{\mathbb R}
\newcommand{\E}{\mathbb E}
\newcommand{\Pp}{\mathbb P}
\newcommand{\cN}{\mathcal N}
\newcommand{\cM}{\mathcal M}
\newcommand{\cZ}{\mathcal Z}
\newcommand{\cL}{\mathcal L}
\newcommand{\cI}{\mathcal I}
\newcommand{\cK}{\mathcal K}
\newcommand{\cG}{\mathcal G}

\newcommand{\tr}{\operatorname{tr}}
\newcommand{\rank}{\operatorname{rank}}
\newcommand{\diag}{\operatorname{diag}}
\newcommand{\proj}{\operatorname{proj}}
\newcommand{\opnorm}[1]{\left\lVert #1\right\rVert_{\mathrm{op}}}
\newcommand{\fnorm}[1]{\left\lVert #1\right\rVert_{\mathrm F}}
\newcommand{\norm}[1]{\left\lVert #1\right\rVert}
\newcommand{\ip}[2]{\left\langle #1,#2\right\rangle}
\newcommand{\ball}{\mathbb B}

\newcommand{\Reg}{\mathfrak R}

\title[Controller Identifiability and Local Minimax Regret]{Controller Identifiability: The Boundary Between Logarithmic and Square-Root Local Minimax Regret in Structured Adaptive LQR}

\author{Zhaobo Liu}
\address{Institute for Advanced Study, Shenzhen University, 3883 Baishi Road, Nanshan District, Shenzhen 518060, China}
\email{liuzhaobo@szu.edu.cn}
\thanks{This work was supported in part by the National Natural Science Foundation of China under Grant 12401585, the Guangdong Basic and Applied Basic Research Foundation under Grant 2024A1515011542, and the General Program of Shenzhen Natural Science Foundation under Grant JCYJ20250604181037012.}
\date{}
\hypersetup{pdftitle={Controller Identifiability: The Boundary Between Logarithmic and Square-Root Local Minimax Regret in Structured Adaptive LQR},pdfauthor={Zhaobo Liu}}

\begin{document}
\begin{abstract}
We study when logarithmic regret is attainable in structured adaptive linear--quadratic regulation. The state and input matrices depend affinely on a common unknown parameter. We consider neighborhoods of a known nominal model that is stabilizable and observable through the output associated with the state cost, with radii proportional to the inverse fourth root of the horizon. The key condition is controller identifiability, which requires the optimal gain derivative to vanish in every parameter direction that leaves the closed-loop dynamics unchanged under nominal optimal feedback. Under suitable regularity conditions on the noise density, this condition at the nominal parameter yields logarithmic local minimax regret when the gain derivative there is nonzero. Failure of the condition yields square-root local minimax regret. If the nominal gain derivative vanishes, local minimax regret remains bounded. We prove that controller identifiability is equivalent to the optimal gain remaining unchanged under sufficiently small invisible perturbations of the nominal model. Motivated by this invariance, we construct a certainty-equivalent policy that estimates only the visible parameter component and attains the logarithmic upper bound under controller identifiability. The policy requires neither the horizon nor the noise law, and its guarantee holds for independent, identically distributed noise with zero mean and finite positive definite covariance.
\end{abstract}

\keywords{
Adaptive control, certainty equivalence, controller identifiability, linear--quadratic regulator, local minimax regret.
}

\maketitle

\section{Introduction}\label{sec:introduction}
Adaptive linear--quadratic regulation (LQR) seeks nearly optimal control of systems with unknown dynamics using information gathered from the controlled trajectory. Classical work established stability and asymptotic performance guarantees for self-tuning regulation and adaptive LQ control \cite{lai1987adaptive,campi1998adaptive}. Finite-horizon regret measures the expected cumulative cost incurred beyond that of an optimal controller that knows the true dynamics \cite{abbasi2011regret,simchowitz2020naive}. Each control input affects both this cost and the information available for future decisions. Inputs that improve estimation may increase the immediate cost, while feedback chosen for control performance may leave some model differences unobserved.

Finite-time regret analyses include robust adaptive control \cite{dean2018regret} and posterior sampling \cite{abeille2018improved,ouyang2020posterior}. For the dependence on the horizon, optimistic control achieves square-root regret up to logarithmic factors under suitable stability and initialization assumptions \cite{abbasi2011regret,cohen2019learning}. Certainty equivalence applies the optimal controller for the estimated model. Perturbation analysis relates its performance to model error \cite{mania2019certainty}, and suitable exploration yields the same square-root dependence \cite{simchowitz2020naive,jedra2022minimal}. Matching lower bounds establish this dependence for the general class with both dynamics matrices unknown \cite{simchowitz2020naive}.

When one dynamics matrix is known, expected regret bounds quadratic in the logarithm of the horizon have been established \cite{cassel2020logarithmic}. With a known stabilizing controller, logarithmic expected regret is attainable when the input matrix is known. When only the state matrix is known, the guarantee additionally requires the optimal gain to have full row rank \cite[Theorems~2--3]{jedra2022minimal}. With both matrices uncertain, a condition relating closed-loop estimation error to error in the optimal gain yields an asymptotic bound quadratic in the logarithm for regret measured against an oracle trajectory \cite{faradonbeh2020adaptive}. With a supplied dynamics representation and stabilizing controller, excitation conditions under the initial and optimal controllers yield an expected regret bound quadratic in the logarithm of the horizon. Sufficiently small representation error adds a term linear in the horizon and quadratic in the representation error \cite[Theorem~3]{lee2024nonasymptotic}. These results identify settings in which additional structure permits faster regret rates. The remaining question is which structural property determines whether logarithmic regret is attainable.

Related regret phase transitions have been studied in a memoryless linear quadratic control model \cite{ziemann2021phase}, while information-based lower bounds cover smooth joint uncertainty and partially observed systems \cite{ziemann2025lower}. For the question studied here, the relevant obstruction is the failure of optimal feedback to reveal model changes that affect the optimal gain. With the state-transition matrix known and the input matrix affinely uncertain, changes invisible under optimal feedback can nevertheless change the optimal gain, forcing a local square-root regret lower bound \cite{ziemann2021uninformative}. That work explicitly asks whether the absence of this obstruction is sufficient for logarithmic regret.

We study this question in structured LQR models whose state-transition and input matrices depend affinely on a common unknown parameter. If systems indistinguishable under nominal optimal feedback share the nominal optimal controller, distinguishing them is unnecessary for optimal control. The relevant distinction is therefore between identifying the entire model and determining its optimal feedback, as recognized in geometric analyses of adaptive LQ control \cite{polderman1986structure,polderman1986necessity}. Data informativity theory likewise studies when systems consistent with noiseless observations share an optimal LQR controller \cite{vanwaarde2020data}. We formulate a criterion, called controller identifiability, to determine whether the information missing under optimal feedback matters for control. In our affine family, we prove that this criterion holds exactly when all sufficiently nearby systems indistinguishable from the nominal model under its optimal feedback share its optimal controller.

A known stabilizable nominal model provides an initial feedback controller that stabilizes all sufficiently nearby systems. We ask how much online adaptation can improve upon this nominal controller when the true dynamics are unknown. Our main classification considers neighborhoods whose radii are proportional to the inverse fourth root of the horizon. At this scale, smooth gain variation and quadratic control loss give a square-root regret upper bound for nominal feedback. We characterize the smallest worst-case expected regret over all causal policies, including randomized ones, and determine when logarithmic regret is attainable. When controller identifiability holds at the nominal model and its state-transition matrix is nonsingular, we also establish logarithmic regret on sufficiently small fixed neighborhoods.

Our main contributions are as follows.

1) We characterize the local minimax regret rate through controller identifiability. Under the stated system and noise assumptions, controller identifiability is necessary and sufficient for logarithmic regret when the nominal gain derivative is nonzero. Its failure yields the square-root rate. A zero nominal gain derivative permits bounded regret. The lower bounds require regularity of the noise location family and apply to all causal policies. For a known state-transition matrix, this resolves the local attainability question in \cite{ziemann2021uninformative} under assumptions common to that work and ours.

2) We construct a certainty-equivalent algorithm that estimates closed-loop dynamics and uses the known affine structure to estimate the parameter component visible under nominal optimal feedback. Under controller identifiability, it attains uniform logarithmic expected regret. It requires neither the horizon nor the noise law, and its guarantee holds for independent, identically distributed noise with zero mean and finite positive definite covariance.

3) We characterize controller identifiability through an equivalent local bound relating changes in the optimal gain to changes in closed-loop dynamics under nominal optimal feedback. We also bound the gain error caused by omitting the invisible parameter component when both visible and invisible components are present. These bounds justify estimating only the visible component in the algorithm.

Section~\ref{sec:problem} defines the model, local experiment, and controller identifiability. Section~\ref{sec:main-results} states the phase classification, establishes its structural characterization, and gives extensions to other uncertainty radii. Section~\ref{sec:FCE} gives the algorithm and its regret bound. Section~\ref{sec:lower} proves the matching lower bounds. Sections~\ref{sec:illustrations} and \ref{sec:conclusion} present illustrations and discuss the scope of the results.

\emph{Notation.} We use \(\norm{\cdot}\) for the Euclidean vector norm or matrix operator norm (also \(\opnorm{\cdot}\)), \(\fnorm{\cdot}\) for the Frobenius norm, and \(\norm{z}_S^2\triangleq z^\top Sz\) for \(S\succ0\). Vector and matrix spaces carry the Euclidean and Frobenius inner products, respectively, with \(\ip{X}{Y}_F\triangleq \tr(X^\top Y)\). Adjoints and Moore--Penrose pseudoinverses \(L^\dagger\) use these inner products. The symbols \(\ker L\), \(M^\perp\), \(L|_M\), and \(\proj_C\) denote the nullspace, orthogonal complement, restriction to a subspace \(M\), and orthogonal projection onto a closed convex set \(C\). The closed unit ball in parameter space is \(\ball_2\). We use \(I_m\) (or \(I\)) for the identity matrix, \(X^*\) for conjugate transpose, \(\lambda_{\min},\lambda_{\max}\) for extreme eigenvalues, and \(\sigma_{\min}\) for the smallest singular value. Probability and expectation are \(\Pp,\E\), and \(\sigma(\cdot)\) denotes the generated sigma-field. For nonnegative quantities \(a,b\), with \(b>0\), we write \(a=O(b)\) if \(a\le Cb\) eventually in the stated limit, and \(a=o(b)\) if \(a/b\to0\). We write \(a=\Theta(b)\), or \(a\asymp b\), if \(cb\le a\le Cb\) eventually, with constants \(0<c\le C<\infty\) independent of the limiting variable. Generic constants \(c,C\) may vary between occurrences. For locally integrable \(f:\R^m\to\R\) and \(g:\R^m\to\R^m\), \(g=\nabla f\) denotes the weak gradient if \(\int_{\R^m}f\,\partial_i\varphi\,dx=-\int_{\R^m}g_i\varphi\,dx\) for every smooth, compactly supported \(\varphi:\R^m\to\R\) and \(i=1,\ldots,m\), where \(\partial_i=\partial/\partial x_i\).

\section{Problem Formulation}\label{sec:problem}
\subsection{System model and optimal control}
Consider the discrete-time linear system
\begin{equation*}
 x_{t+1}=A_\theta x_t+B_\theta u_t+w_{t+1},\qquad x_0\in\R^{n_x}, 
\end{equation*}
where \(t=0,1,\ldots\), \(x_t\in\R^{n_x}\) is the state, \(u_t\in\R^{n_u}\) is the control input, and \(w_{t+1}\in\R^{n_x}\) is process noise. The state is observed without measurement noise, and the initial state \(x_0\) is fixed and known. The dynamics depend on an unknown parameter \(\theta\in\Theta\subseteq\R^d\) through the affine structure
\begin{equation*}
 A_\theta=A_0+\sum_{i=1}^d\theta_iA_i,
 \qquad
 B_\theta=B_0+\sum_{i=1}^d\theta_iB_i.
\end{equation*}
where \(A_i\in\R^{n_x\times n_x}\) and \(B_i\in\R^{n_x\times n_u}\), \(i=0,\ldots,d\), are known matrices. For \(v\in\R^d\), write
\begin{equation*}
 A[v]\triangleq \sum_{i=1}^dv_iA_i,\qquad B[v]\triangleq \sum_{i=1}^dv_iB_i.
\end{equation*}
The stage cost is
\begin{equation*}
 c(x,u)=x^\top Qx+u^\top Ru,
 \qquad Q\succeq0,\quad R\succ0.
\end{equation*}

Let \((\xi_t)_{t\ge0}\) be the policy randomization sequence, independent of the process noise, whose joint distribution does not depend on \(\theta\). The variable \(\xi_t\) is available at time \(t\). Let \(\mathcal F_t\triangleq \sigma(x_0,\ldots,x_t,u_0,\ldots,u_{t-1},\xi_0,\ldots,\xi_t)\), \(t\ge0\), with the input list empty at \(t=0\).
A causal policy \(\pi\) chooses \(u_t\) as an \(\mathcal F_t\)-measurable function, and \(w_{t+1}\) is independent of \(\mathcal F_t\).

We impose the following moment condition on the process noise.

\begin{assumption}\label{ass:noise-moments}
The disturbances are independent and identically distributed, with a distribution independent of \(\theta\), and satisfy
\begin{equation*}
 \E w_t=0,\qquad \E\norm{w_t}^2<\infty,\qquad
 W\triangleq \E[w_tw_t^\top]\succ0.
\end{equation*}
\end{assumption}

\begin{assumption}\label{ass:regular}
The parameter domain \(\Theta\) contains an open neighborhood of a known point \(\theta_0\). The cost matrices \(Q,R\) are known, \((A_{\theta_0},B_{\theta_0})\) is stabilizable, and \((Q^{1/2},A_{\theta_0})\) is observable.
\end{assumption}

Observability means that every nonzero initial state produces positive state cost at some time under the nominal dynamics with zero input and noise. It holds automatically when \(Q\succ0\). Both system properties persist in a sufficiently small nominal neighborhood (\cref{lem:benchmark}).

For \(\theta\in\Theta\) such that \((A_\theta,B_\theta)\) is stabilizable and \((Q^{1/2},A_\theta)\) is observable, let \(P_\theta\succ0\) solve the discrete algebraic Riccati equation (DARE) \cite[Theorems~5--6]{hager1976convergence}. Set
\begin{equation}\label{eq:optimal-quantities}
\begin{aligned}
 S_\theta&\triangleq R+B_\theta^\top P_\theta B_\theta,
 & K_\theta&\triangleq S_\theta^{-1}B_\theta^\top P_\theta A_\theta,\\
 F_\theta&\triangleq A_\theta-B_\theta K_\theta,
 & J_\theta&\triangleq \tr(P_\theta W).
\end{aligned}
\end{equation}
The DARE has the Bellman form
\begin{equation}
 P_\theta=Q+K_\theta^\top RK_\theta+F_\theta^\top P_\theta F_\theta.\label{eq:bellman-riccati}
\end{equation}
The stationary optimal feedback is \(u=-K_\theta x\), with closed-loop matrix \(F_\theta\) and average cost \(J_\theta\).

\subsection{Regret and local minimax formulation}
Let \(\Pp_\theta^\pi,\E_\theta^\pi\) denote trajectory probability and expectation under parameter \(\theta\) and policy \(\pi\), over process noise and policy randomization. Indices fixed by context may be omitted.
For an integer horizon \(T\ge1\), let
\begin{equation*}
 V_T^\pi(\theta)\triangleq \E_\theta^\pi\sum_{t=0}^{T-1}c(x_t,u_t),
 \qquad
 V_T^\star(\theta)\triangleq \inf_\pi V_T^\pi(\theta).
\end{equation*}
The infimum is over all causal policies with knowledge of \(\theta\). No terminal cost is included. We use finite-horizon regret
\begin{equation}
 R_T^\pi(\theta)\triangleq V_T^\pi(\theta)-V_T^\star(\theta).         \label{eq:FH-regret}
\end{equation}
To compare systems near the nominal point, for a fixed \(\varepsilon>0\) define
\begin{equation}
 \Theta_T(\theta_0,\varepsilon)
 \triangleq \{\theta_0+h:\norm{h}\le\varepsilon T^{-1/4}\}.           \label{eq:local-class}
\end{equation}
We consider sufficiently large \(T\) so that this ball lies in the nominal neighborhood specified above.

Let \(\Pi_T(\theta_0,\varepsilon)\) be all causal policies allowed to know \(T,\theta_0,\varepsilon\), the structural model, and the fixed noise law, but not the true local alternative. Each policy uses the same decision rules across that ball. Define
\begin{equation}
 \Reg_T(\theta_0,\varepsilon)
 \triangleq \inf_{\pi\in\Pi_T(\theta_0,\varepsilon)}
 \sup_{\theta\in\Theta_T(\theta_0,\varepsilon)}R_T^\pi(\theta). \label{eq:minimax}
\end{equation}

We use the local scale of \cite{ziemann2021uninformative}. On this neighborhood, nominal optimal feedback has \(O(\sqrt T)\) regret by \cref{prop:nominal-upper}.

We also use regret relative to the optimal average cost,
\begin{equation*}
 \bar R_T^\pi(\theta)\triangleq V_T^\pi(\theta)-TJ_\theta.
\end{equation*}
By \cref{lem:benchmark}, the two regret measures differ by a uniformly bounded constant for policies with finite expected cumulative cost.

For the lower bounds, we use an exact identity that expresses regret as a sum of nonnegative control losses. Let \(P_{\theta,0}=0\), and let \(P_{\theta,h}\) and \(K_{\theta,h}\) denote the Riccati solution and gain with \(h\ge1\) steps remaining, defined by \eqref{eq:app-fh-rec}. Set \(S_{\theta,h}\triangleq R+B_\theta^\top P_{\theta,h-1}B_\theta\). For a policy with finite \(V_T^\pi(\theta)\), the finite-horizon performance difference identity is
\begin{equation}
 R_T^\pi(\theta)
 =\sum_{t=0}^{T-1}\E_\theta^\pi
 \norm{u_t+K_{\theta,T-t}x_t}_{S_{\theta,T-t}}^2.             \label{eq:FH-PDI}
\end{equation}
This identity is proved in Appendix~\ref{app:nominal-bounds}.

\subsection{Controller identifiability and Fisher information}
Models indistinguishable under a fixed optimal feedback may require different optimal gains. Controller identifiability asks whether model changes invisible under optimal feedback also leave the optimal gain unchanged to first order.

Fix an interior parameter \(\theta\in\Theta\) with \((A_\theta,B_\theta)\) stabilizable and \((Q^{1/2},A_\theta)\) observable. Hold \(u=-K_\theta x\) fixed while changing the plant parameter to \(\theta+sv\), where \(v\in\R^d\) and \(s\in\R\) is small. The closed-loop matrix changes by \(s\cZ_\theta v\), where \(\cZ_\theta v\triangleq A[v]-B[v]K_\theta\). Directions in \(\ker\cZ_\theta\) leave the closed-loop dynamics unchanged and are called invisible under this feedback.

The gain derivative is the linear map \(DK_\theta:\R^d\to\R^{n_u\times n_x}\) defined by \(DK_\theta[v]\triangleq \left.\frac{d}{ds}K_{\theta+sv}\right|_{s=0}\). The notation \(DK_\theta=0\) means that it vanishes in every direction.

\begin{definition}\label{def:CI}
Controller identifiability (CI) holds at \(\theta\) if
\begin{equation*}
 \ker\cZ_\theta\subseteq\ker DK_\theta.
\end{equation*}
\end{definition}

Thus CI requires that every invisible direction have zero derivative of the optimal gain. In the local minimax problem, CI is evaluated at \(\theta_0\).

To express invisibility in terms of Fisher information, we additionally impose regularity on the noise density.

\begin{assumption}\label{ass:location-noise}
The noise law has a Lebesgue density \(p\), and \(f=\sqrt p\) has a weak gradient satisfying
\begin{equation*}
 \int_{\R^{n_x}}\norm{\nabla f(w)}^2\,dw<\infty.
\end{equation*}
\end{assumption}

Assumption~\ref{ass:location-noise} applies to Fisher information statements and lower bounds. Define the location score and Fisher information matrix by
\begin{equation*}
\begin{gathered}
 \psi(w)\triangleq -2\frac{\nabla f(w)}{f(w)},\\
 J_w\triangleq \E[\psi(w)\psi(w)^\top]=4\int \nabla f(w)\nabla f(w)^\top\,dw,
\end{gathered}
\end{equation*}
where \(\psi\) is defined arbitrarily on \(\{f=0\}\). For a positive differentiable density, \(\psi=-\nabla\log p\). For Gaussian noise, \(\psi(w)=W^{-1}w\) and \(J_w=W^{-1}\).

Invisible directions are precisely those with zero Fisher information under the fixed optimal feedback. Let \(\Sigma_\theta\succ0\) be its stationary state covariance, satisfying \(\Sigma_\theta=F_\theta\Sigma_\theta F_\theta^\top+W\). Under \cref{ass:location-noise}, averaging the single-step directional Fisher information over the stationary state distribution gives
\begin{equation*}
 \cI_\theta^\star(v,v)
 =\tr\!\left[J_w(\cZ_\theta v)\Sigma_\theta(\cZ_\theta v)^\top\right]. 
\end{equation*}
Since \(J_w\succ0\) and \(\Sigma_\theta\succ0\),
\begin{equation}
 \cI_\theta^\star(v,v)=0\quad\text{if and only if}\quad\cZ_\theta v=0.                       \label{eq:Fisher-kernel}
\end{equation}
Thus CI requires the optimal gain derivative to vanish in every direction carrying zero information under the fixed optimal feedback. Appendix~\ref{app:info} derives the information formula and establishes \(J_w\succ0\).

\section{Main Results}\label{sec:main-results}
\subsection{Local minimax classification}
\begin{theorem}\label{thm:phase}
Under Assumptions~\ref{ass:noise-moments}--\ref{ass:location-noise}, fix \(\varepsilon\in(0,\infty)\). In each case below, there are constants \(0<c\le C<\infty\) such that, for all sufficiently large \(T\):
\begin{enumerate}
\item If \(DK_{\theta_0}=0\), then
\begin{equation*}
 \Reg_T(\theta_0,\varepsilon)\le C.
\end{equation*}
\item If CI holds at \(\theta_0\) and \(DK_{\theta_0}\neq0\), then
\begin{equation*}
 c\log T\le \Reg_T(\theta_0,\varepsilon)
 \le C\log T.
\end{equation*}
\item If CI fails at \(\theta_0\), then
\begin{equation*}
 c\sqrt T\le \Reg_T(\theta_0,\varepsilon)
 \le C\sqrt T.
\end{equation*}
\end{enumerate}
All upper bounds remain valid without \cref{ass:location-noise} and are attained by policies that do not use the noise law. The constants may depend on that fixed law, the nominal instance, \(x_0\), and \(\varepsilon\), but not on \(T\).
\end{theorem}

The upper bounds are proved in Section~\ref{sec:FCE}. The matching lower bounds and the proof of \cref{thm:phase} are completed in Section~\ref{sec:lower}.

When \(A\) is known and \(B\) depends affinely on the unknown parameter,
\begin{equation*}
 \cZ_\theta v=-B[v]K_\theta.
\end{equation*}
By \eqref{eq:Fisher-kernel}, the uninformativeness criterion in \cite{ziemann2021uninformative} is equivalent to failure of CI.

\begin{corollary}\label{cor:ZS}
Consider the subclass with known \(A\) and \(B\) depending affinely on the unknown parameter, under \cref{ass:noise-moments,ass:regular,ass:location-noise}. In the local minimax experiment \eqref{eq:local-class}--\eqref{eq:minimax}, whose nominal center is known, an \(O(\log T)\) regret upper bound is attainable if and only if the nominal optimal policy is not uninformative in the sense of \cite{ziemann2021uninformative}.
\end{corollary}

\subsection{Characterizations of CI}
The following theorem characterizes CI through exact invariance of the optimal controller and a bound on gain error.

\begin{theorem}\label{thm:quotient}
Fix an interior parameter \(\theta\) with \((A_\theta,B_\theta)\) stabilizable and \((Q^{1/2},A_\theta)\) observable. Then the following are equivalent.
\begin{enumerate}
\item CI holds at \(\theta\), i.e., \(\ker\cZ_\theta\subseteq\ker DK_\theta\).
\item There is a neighborhood of \(\theta\) in \(\Theta\) on which \((A_{\theta'},B_{\theta'})\) is stabilizable and \((Q^{1/2},A_{\theta'})\) is observable. For every \(\theta'\) in this neighborhood with \(\theta'-\theta\in\ker\cZ_\theta\), the Riccati solution, optimal gain, and optimal closed-loop matrix equal \(P_\theta\), \(K_\theta\), and \(F_\theta\), respectively.
\item There are \(r,C>0\) such that the closed ball \(\{\theta':\norm{\theta'-\theta}\le r\}\) lies in \(\Theta\) and consists of stabilizable systems observable through \(Q^{1/2}\). For every \(\theta'\) in this ball,
\begin{equation*}
 \norm{K_{\theta'}-K_\theta}
 \le C\fnorm{\cZ_\theta(\theta'-\theta)}.
\end{equation*}
\end{enumerate}
\end{theorem}

The proof is in Appendix~\ref{app:quotient-proof}.

\begin{remark}
Under CI at \(\theta\), gain error is controlled by the change in closed-loop dynamics under the fixed feedback \(K_\theta\), rather than by the full plant error.
\end{remark}

The following corollary quantifies how much plant uncertainty can remain under CI. Let
\(\mathcal D\triangleq\{v:A[v]=0,\ B[v]=0\}\)
denote parameter changes that leave both dynamics matrices unchanged, and define
\begin{equation}
 \cN_\theta\triangleq \ker\cZ_\theta,
 \qquad
 \cM_\theta\triangleq \cN_\theta^\perp.                              \label{eq:split-theta}
\end{equation}

\begin{corollary}\label{cor:physical-identifiability}
Fix an interior parameter \(\theta\) with \((A_\theta,B_\theta)\) stabilizable and \((Q^{1/2},A_\theta)\) observable. If CI holds at \(\theta\), then
\begin{equation}
 \dim\cN_\theta-\dim\mathcal D
 \le n_u\bigl(n_x-\rank A_\theta\bigr).                 \label{eq:physical-nullity}
\end{equation}
If \(A_\theta\) is nonsingular, CI at \(\theta\) is equivalent to \(\cN_\theta=\mathcal D\), that is, plant identifiability under optimal feedback after redundant parameters are removed.
\end{corollary}

The proof is in Appendix~\ref{app:physical-proof}.

\begin{remark}
The difference \(\dim\cN_\theta-\dim\mathcal D\) counts independent invisible plant changes after removing parameter redundancy. By \eqref{eq:physical-nullity}, CI at \(\theta\) can hold without plant identifiability under optimal feedback only if \(A_\theta\) is singular.
For unrestricted dynamics matrices, the dimension of the set sharing an optimal feedback and closed-loop matrix is related to the rank of the state matrix \cite[Theorem~2]{faradonbeh2020adaptive}.
\end{remark}

Section~\ref{sec:auxiliary-example} gives an example in which an unknown input coefficient changes the plant but not its optimal controller.

\subsection{Uncertainty radius and fixed neighborhoods}\label{sec:radius}
The neighborhoods in \cref{thm:phase} shrink at rate \(T^{-1/4}\). We now examine how regret depends on the radius and give sufficient conditions for logarithmic regret on a fixed neighborhood. For \(r>0\) sufficiently small that \(\theta_0+r\ball_2\subseteq\Theta\), define
\begin{equation*}
 \begin{aligned}
 \Theta(r)&\triangleq\theta_0+r\ball_2,\\
 \mathcal R_T(r)&\triangleq\inf_\pi\sup_{\theta\in\Theta(r)}R_T^\pi(\theta).
 \end{aligned}
\end{equation*}
The policies know \(\theta_0,r,T\) and the same structural and noise information as before.

\begin{proposition}\label{thm:radius}
Under \cref{ass:noise-moments,ass:location-noise,ass:regular}, there are \(r_0,c,C>0\) and \(T_0<\infty\), independent of \(r,T\), such that for \(0<r\le r_0\) and \(T\ge T_0\), the nominal policy \(u_t=-K_{\theta_0}x_t\) gives \(\mathcal R_T(r)\le C(1+Tr^2)\). In the following bounds, CI is evaluated at \(\theta_0\):
\begin{align}
 \mathcal R_T(r)&\le C\bigl[1+\log(1+Tr^2)+Tr^4\bigr]
 &&\text{if CI holds}, \label{eq:radius-CI}\\
 \mathcal R_T(r)&\ge c\log(1+Tr^2)-C
 &&\text{if }DK_{\theta_0}\ne0, \label{eq:radius-log}\\
 \mathcal R_T(r)&\ge c\min\{Tr^2,\sqrt T\}-C
 &&\text{if CI fails}. \label{eq:radius-hard}
\end{align}
The \(Tr^4\) term in \eqref{eq:radius-CI} can be omitted if \(\cN_{\theta_0}=\mathcal D\). If \(DK_{\theta_0}=0\), nominal control instead gives \(\mathcal R_T(r)\le C(1+Tr^4)\). All upper bounds hold without \cref{ass:location-noise}, using policies that do not use the noise law.
\end{proposition}

Appendix~\ref{app:radius} proves the bounds and uniformity of their constants. The policy attaining \eqref{eq:radius-CI} is constructed in Remark~\ref{rem:radius-policy}.

Setting \(r=\varepsilon T^{-1/4}\) recovers the bounded, logarithmic, and square-root rates in \cref{thm:phase}. For fixed \(r>0\), however, the \(Tr^4\) term grows linearly with \(T\), so \eqref{eq:radius-CI} alone does not give logarithmic regret. The \(Tr^4\) term bounds the loss caused by invisible model variation. If \(A_{\theta_0}\) is nonsingular and CI holds at \(\theta_0\), \cref{cor:physical-identifiability} gives \(\cN_{\theta_0}=\mathcal D\), which removes this term and yields the following consequence.

\begin{corollary}\label{cor:fixed-radius}
Under \cref{ass:noise-moments,ass:regular}, suppose \(A_{\theta_0}\) is nonsingular and CI holds at \(\theta_0\). For every sufficiently small fixed \(r>0\), there exists a policy with \(O(\log T)\) regret uniformly over \(\Theta(r)\). The policy uses the known center and radius, but neither \(T\) nor the noise law. If, additionally, \cref{ass:location-noise} holds and \(DK_{\theta_0}\ne0\), then \(\mathcal R_T(r)=\Theta(\log T)\).
\end{corollary}

The proof is in Appendix~\ref{app:radius}.

\section{Certainty Equivalence and Regret Upper Bounds}\label{sec:upper-simple}\label{sec:FCE}
This section proves all upper bounds in \cref{thm:phase}. For the remainder, write \(K_0\triangleq K_{\theta_0}\). A gain \(K\) in the policy superscript of \(V_T^K\) or \(R_T^K\) denotes the stationary policy \(u_t=-Kx_t\).

Under \cref{ass:noise-moments,ass:regular}, the fixed policy \(u_t=-K_0x_t\) gives \(O(\sqrt T)\) regret, or \(O(1)\) if \(DK_{\theta_0}=0\), uniformly over \(\Theta_T(\theta_0,\varepsilon)\) for fixed \(\varepsilon>0\) and sufficiently large \(T\) (\cref{prop:nominal-upper} in Appendix~\ref{app:nominal-bounds}). We therefore focus on constructing a policy that improves the bound to \(O(\log T)\) when CI holds at \(\theta_0\) and \(DK_{\theta_0}\neq0\).

\subsection{Fiber-aware certainty equivalence}
We call the policy in Algorithm~\ref{alg:FCE} fiber-aware certainty equivalence (FCE). Here fibers are sets of parameters that give the same closed-loop matrix under nominal optimal feedback. Under CI at \(\theta_0\), FCE estimates only the visible parameter component. Decompose a true local alternative as
\begin{equation}
 \theta-\theta_0=\mu+\nu,
 \qquad \mu\in\cM_{\theta_0},\quad\nu\in\cN_{\theta_0}.                       \label{eq:true-split}
\end{equation}
The components \(\mu\) and \(\nu\) are visible and invisible under this feedback, respectively. If \(\cM_{\theta_0}=\{0\}\), all sufficiently nearby parameters share the nominal optimal gain (\cref{thm:exact-fiber}), so nominal control gives bounded regret without estimation (\cref{prop:nominal-upper}). Henceforth assume \(\cM_{\theta_0}\neq\{0\}\).

Starting from \(\widehat\mu_0=0\), FCE applies the optimal feedback for the model \(\theta_0+\widehat\mu_j\) during epoch \(j=0,1,\ldots\). The estimate \(\widehat\mu_j\in\cM_{\theta_0}\) is held fixed throughout the epoch, with
\begin{equation*}
 K_j\triangleq K_{\theta_0+\widehat\mu_j},
 \qquad u_t=-K_jx_t.
\end{equation*}
Fix any integer \(N_0\ge4\) independently of the noise law. Epoch \(j\) starts at \(\tau_j\triangleq N_0(2^j-1)\) and lasts \(N_j=N_02^j\) steps, so the time between updates doubles at each epoch.

The update estimates \(\mu\) from the change in closed-loop dynamics under the current feedback. For any feedback \(K\), define
\begin{equation*}
 \cZ_Kv\triangleq A[v]-B[v]K,
 \qquad
 \cL_K\triangleq \cZ_K|_{\cM_{\theta_0}}.
\end{equation*}
Here \(\cZ_{K_\theta}=\cZ_\theta\). The restriction \(\cL_{K_0}\) is injective, and \cref{lem:left-inverse} gives uniformly bounded left inverses for nearby \(\cL_K\). Choose \(\bar\rho>0\) small enough to preserve this injectivity and local stability (\cref{lem:common-lyap}). The parameter projection keeps \(\theta_0+\widehat\mu_j\) in this neighborhood. Set \(\rho_0\triangleq \bar\rho\), choose \(c_\rho>\varepsilon\), and use the projection radius
\begin{equation}
 \rho_{j+1}\triangleq \min\{\bar\rho,c_\rho N_j^{-1/4}\}.             \label{eq:rho}
\end{equation}

To compute the update, compare the estimated dynamics with the nominal model under the same feedback \(K_j\). The nominal closed-loop matrix is \(G_{0,j}\triangleq A_{\theta_0}-B_{\theta_0}K_j\).
For the matrix estimate, fix in advance a compact convex Frobenius ball \(\cG\), known to the learner, containing \(A_\theta-B_\theta K_{\theta_0+\eta}\) for \(\norm{\theta-\theta_0}\le\bar\rho\) and \(\eta\in\cM_{\theta_0}\) with \(\norm{\eta}\le\bar\rho\).

At the end of epoch \(j\), estimate \(G_{\theta,j}\triangleq A_\theta-B_\theta K_j\) by ordinary least squares (OLS) using \((x_t,x_{t+1})\) for \(t=\tau_j+1,\ldots,\tau_j+N_j-1\), as in \eqref{eq:app-OLSdef}. Project the OLS estimate onto \(\cG\) in the Frobenius norm to obtain \(\widehat G_j\). Subtract \(G_{0,j}\), apply the restricted left inverse, and project the resulting parameter estimate to obtain
\begin{align}
 \widetilde\mu_{j+1}
 &\triangleq \cL_{K_j}^\dagger(\widehat G_j-G_{0,j}),                 \label{eq:mu-raw}\\
 \widehat\mu_{j+1}
 &\triangleq \proj_{\cM_{\theta_0}\cap\rho_{j+1}\ball_2}(\widetilde\mu_{j+1}). \label{eq:mu-proj}
\end{align}
\begin{algorithm}[t]
\caption{Fiber-Aware Certainty Equivalence}
\label{alg:FCE}
\begin{algorithmic}[1]
\Require \(\theta_0\), structural maps, \(\varepsilon\), \(N_0\), \(c_\rho\), maximum projection radius \(\bar\rho\), matrix projection ball \(\cG\)
\State compute \(K_0\), \(\cN_{\theta_0}=\ker\cZ_{K_0}\), \(\cM_{\theta_0}=\cN_{\theta_0}^\perp\)
\State set \(\widehat\mu_0=0\)
\For{\(j=0,1,2,\ldots\)}
  \State set \(N_j=N_02^j\), \(K_j=K_{\theta_0+\widehat\mu_j}\)
  \State apply \(u_t=-K_jx_t\) for \(N_j\) steps
  \State compute OLS using the last \(N_j-1\) transitions of the epoch and project onto \(\cG\) to obtain \(\widehat G_j\), as in \eqref{eq:app-OLSdef}
  \State compute \(\widetilde\mu_{j+1}\) by \eqref{eq:mu-raw}
  \State set \(\rho_{j+1}\) by \eqref{eq:rho} and obtain \(\widehat\mu_{j+1}\) by projection \eqref{eq:mu-proj}
\EndFor
\end{algorithmic}
\end{algorithm}

\begin{remark}\label{rem:radius-policy}
For the radius bounds in \cref{thm:radius,cor:fixed-radius}, use Algorithm~\ref{alg:FCE} with fixed projection onto \(\cM_{\theta_0}\cap r\ball_2\) and initial epoch length \(N_0(r)=\lceil4r^{-2}\rceil\). The radius \(r\) is an input, and the resulting policy is anytime. Appendix~\ref{app:radius} establishes its regret bounds.
\end{remark}

\subsection{Estimation and gain error}
The gain error of FCE depends on both estimation of the visible component \(\mu\) and the effect of the invisible component \(\nu\). We first bound the gain error from omitting \(\nu\) when \(\mu\) is known.

\begin{proposition}\label{thm:cross-fiber}
Under \cref{ass:regular}, if CI holds at \(\theta_0\), then there are \(r_2,L_2>0\) such that the closed \(r_2\)-ball about \(\theta_0\) lies in the parameter domain and consists of stabilizable systems observable through \(Q^{1/2}\), and
\begin{equation}
 \norm{K_{\theta_0+\mu+\nu}-K_{\theta_0+\mu}}
 \le L_2\norm{\mu}\norm{\nu}                              \label{eq:cross-fiber}
\end{equation}
for all \(\mu\in\cM_{\theta_0}\), \(\nu\in\cN_{\theta_0}\) with \(\norm{\mu}+\norm{\nu}\le r_2\).  The same estimate holds for \(P_\theta\).
\end{proposition}

The proof is in Appendix~\ref{app:cross-fiber-proof}.

If \(\mu\) were known, ignoring \(\nu\) would mean using the gain \(K_{\theta_0+\mu}\). On \(\Theta_T(\theta_0,\varepsilon)\), both components are \(O(T^{-1/4})\), so \eqref{eq:cross-fiber} gives a gain error of \(O(T^{-1/2})\).

The invisible component also affects estimation of \(\mu\) under the current feedback. By \eqref{eq:true-split},
\begin{equation}
 G_{\theta,j}=G_{0,j}+\cZ_{K_j}\mu+\cZ_{K_j}\nu.            \label{eq:Gsplit}
\end{equation}
The term \(\cZ_{K_j}\nu\) can bias the estimate of \(\mu\). Since \(\cZ_{K_0}\nu=0\),
\begin{equation}
 \cZ_{K_j}\nu=-B[\nu](K_j-K_0).                     \label{eq:nuisance-identity}
\end{equation}
Thus the perturbation of the regression equation from \(\nu\) is bounded by a constant times \(\norm{\nu}\norm{K_j-K_0}\). Keeping the controller near \(K_0\) limits this source of estimation error.

Write \(r_T\triangleq \varepsilon T^{-1/4}\) for the radius of the local parameter set \(\Theta_T(\theta_0,\varepsilon)\).

\begin{proposition}\label{prop:FCE-track}
Under Assumptions~\ref{ass:noise-moments} and~\ref{ass:regular}, suppose CI holds at \(\theta_0\). For FCE, there is a constant \(C>0\), independent of \(j,T\) and the true parameter in \(\Theta_T(\theta_0,\varepsilon)\), such that, for every epoch completed by time \(T\) and all sufficiently large \(T\),
\begin{align}
 \E\norm{\widehat\mu_{j+1}-\mu}^2
 &\le \frac{C}{N_j}+Cr_T^2\rho_j^2,                          \label{eq:mu-track}\\
 \E\norm{K_{\theta_0+\widehat\mu_{j+1}}-K_\theta}^2
 &\le \frac{C}{N_j}+Cr_T^2\rho_j^2+Cr_T^4.                  \label{eq:K-track}
\end{align}
Moreover, for every epoch that starts before time \(T\), \(\norm{K_j-K_\theta}\le C\rho_j\) deterministically on the local class.
\end{proposition}

In \eqref{eq:K-track}, \(C/N_j\) is the statistical estimation term, \(Cr_T^2\rho_j^2\) accounts for the effect of \(\nu\) on the regression, and \(Cr_T^4\) bounds the squared gain error from omitting \(\nu\).

\begin{proof}
For all sufficiently large \(T\), \(r_T\le\bar\rho\). If epoch \(j\) is completed by time \(T\), then \(N_j\le T\). Since \(c_\rho>\varepsilon\),
\begin{equation}
 \norm{\mu}\le r_T
 \le\min\{\bar\rho,c_\rho N_j^{-1/4}\}
 =\rho_{j+1}.                                                \label{eq:truth-in-projection}
\end{equation}
Thus the projection set contains \(\mu\).

From \eqref{eq:Gsplit},
\begin{equation*}
 \widehat G_j-G_{0,j}
 =\cL_{K_j}\mu+\cZ_{K_j}\nu+(\widehat G_j-G_{\theta,j}).
\end{equation*}
The expected squared Frobenius norm of the last term is at most \(C/N_j\) by \cref{lem:OLS}, while \cref{lem:nuisance} bounds \(\fnorm{\cZ_{K_j}\nu}\) by \(Cr_T\rho_j\). Apply \(\cL_{K_j}^\dagger\) and use \(\cL_{K_j}^\dagger\cL_{K_j}\mu=\mu\) in the raw update \eqref{eq:mu-raw}. The bounded left inverse (\cref{lem:left-inverse}) and nonexpansiveness of projection then give \eqref{eq:mu-track}.
To pass from parameter error to gain error, write
\begin{align*}
 K_{\theta_0+\widehat\mu_{j+1}}-K_{\theta_0+\mu+\nu}
 =&\ K_{\theta_0+\widehat\mu_{j+1}}-K_{\theta_0+\mu}\\
 &+K_{\theta_0+\mu}-K_{\theta_0+\mu+\nu},
\end{align*}
where local Lipschitz continuity bounds the first difference by the parameter estimation error, and \cref{thm:cross-fiber} bounds the second by \(L_2\norm{\mu}\norm{\nu}\). Squaring and taking expectations gives \eqref{eq:K-track}.

For the deterministic bound, consider an epoch that starts before \(T\). If \(j=0\), then \(\widehat\mu_0=0\) and \(r_T\le\rho_0=\bar\rho\). Otherwise, epoch \(j-1\) was completed, so \eqref{eq:truth-in-projection} gives \(r_T\le\rho_j\). In both cases, \(\norm{\widehat\mu_j},\norm{\mu},\norm{\nu}\le\rho_j\), and local Lipschitz continuity gives \(\norm{K_j-K_\theta}\le C\rho_j\).
\end{proof}

\subsection{Logarithmic regret}
We write \(R_T^{\mathrm{FCE}}(\theta)\) for the regret \eqref{eq:FH-regret} of the policy in Algorithm~\ref{alg:FCE}.

\begin{theorem}\label{thm:log-upper}
Under \cref{ass:noise-moments,ass:regular}, suppose CI holds at \(\theta_0\). For every fixed \(\varepsilon>0\), there is a constant \(C>0\), independent of \(T\), such that FCE satisfies
\begin{equation*}
 \sup_{\theta\in\Theta_T(\theta_0,\varepsilon)}
 R_T^{\mathrm{FCE}}(\theta)
 \le C\log T
\end{equation*}
for all sufficiently large \(T\). The policy is anytime and does not use the noise law.
\end{theorem}

\begin{proof}
The stationary Bellman identity and \cref{lem:benchmark} bound regret by the cumulative quadratic control loss plus a constant. Let \(e_j\triangleq \norm{K_j-K_\theta}\), and let \(H\succ0\) be the common Lyapunov matrix from \cref{lem:common-lyap}. The gain \(K_j\) is fixed during epoch \(j\) and is \(\mathcal F_{\tau_j}\)-measurable. Applying the conditional bound on state energy \eqref{eq:conditional-energy} and taking expectations gives
\begin{equation}
 \E\sum_{t=\tau_j}^{\tau_j+N_j-1}\norm{(K_j-K_\theta)x_t}_{S_\theta}^2
 \le C N_j\E e_j^2+C\E[e_j^2x_{\tau_j}^\top Hx_{\tau_j}],    \label{eq:epoch-loss}
\end{equation}
where the second term accounts for the state at the start of the epoch. The deterministic bound \(e_j\le C\rho_j\) and \eqref{eq:uniform-state} bound this term by \(C\rho_j^2\), which is summable over \(j\).

It remains to sum \(N_j\E e_j^2\). In the following sums, \(j\) ranges over epochs with \(j\ge1\) completed by time \(T\). The controller in epoch \(j\) uses the estimate from epoch \(j-1\), so \cref{prop:FCE-track} gives
\(
\E e_j^2\le C/N_{j-1}+Cr_T^2\rho_{j-1}^2+Cr_T^4
\).
Since \(N_j=2N_{j-1}\), the statistical term contributes \(O(1)\) in each of the \(O(\log T)\) epochs. The other two terms have bounded sums.
\begin{align*}
 \sum_j N_j/N_{j-1}&=O(\log T),\\
 \sum_j N_jr_T^2\rho_{j-1}^2
 &\le CT^{-1/2}\sum_{N_j\le T}N_j^{1/2}=O(1),\\
 \sum_jN_jr_T^4&\le Tr_T^4=O(1).
\end{align*}
The initial epoch contributes \(O(1)\) because its length is fixed and \eqref{eq:uniform-state} bounds the state second moment. If the horizon ends inside epoch \(j\ge1\), apply \eqref{eq:epoch-loss} to its observed portion and use the tracking bound from the completed epoch \(j-1\). Since \(N_j=2N_{j-1}\le2T\), the same estimates give an \(O(1)\) contribution from this final portion.
\end{proof}

\section{Regret Lower Bounds}\label{sec:lower}
The following proposition provides the lower bounds needed for \cref{thm:phase}.

\begin{proposition}\label{thm:log-lower}\label{thm:sqrt-lower}
Under \cref{ass:noise-moments,ass:location-noise,ass:regular}, fix \(\varepsilon>0\). In each case below there is a constant \(c>0\), independent of \(T\), such that the bound holds for all sufficiently large \(T\).
\begin{enumerate}
\item If \(DK_{\theta_0}\neq0\), then
\begin{equation}
 \Reg_T(\theta_0,\varepsilon)\ge c\log T. \label{eq:log-lower}
\end{equation}
Under CI at \(\theta_0\), this matches \cref{thm:log-upper}.
\item If there is a unit direction \(v\) satisfying
\begin{equation}
 \cZ_{\theta_0}v=0,\qquad DK_{\theta_0}[v]\neq0, \label{eq:harmful-direction}
\end{equation}
then
\begin{equation}
 \Reg_T(\theta_0,\varepsilon)\ge c\sqrt T. \label{eq:sqrt-lower}
\end{equation}
\end{enumerate}
\end{proposition}

The proposition is proved in Section~\ref{sec:lower-proof}, which also completes the proof of \cref{thm:phase}.

\subsection{From control loss to gain estimation}\label{sec:lower-reduction}
Both lower bounds compare the accuracy of estimating the optimal gain with the control loss of the policy generating the data. Small regret forces this estimation error to be small, whereas the available Fisher information limits how small it can be.

We consider policies with finite worst-case expected cumulative cost over \(\Theta_T(\theta_0,\varepsilon)\). Policies with infinite worst-case cost already satisfy the lower bounds, since \(V_T^\star\) is uniformly bounded on this set for each fixed \(T\).

Fix a unit direction \(v\in\R^d\), choose \(0<a<\varepsilon\) small enough that \(\{\theta_0+sv:|s|\le a\}\) remains in the fixed stabilizing neighborhood, let \(\theta_s=\theta_0+sv\), and restrict \(|s|\le aT^{-1/4}\). Write \(K_s=K_{\theta_s}\). For a fixed policy \(\pi\), let \(R_T(s)=R_T^\pi(\theta_s)\), \(\E_s\triangleq \E_{\theta_s}^\pi\), and \(\Pp_s\triangleq \Pp_{\theta_s}^\pi\). For a prior density \(\lambda\) on this scalar line, \(\E_{s\sim\lambda}\) denotes integration over \(s\) with density \(\lambda\). Write \(B_T(\lambda)\triangleq\E_{s\sim\lambda}R_T(s)\) for the resulting average regret.

The controller residual \(u_k+K_sx_k\) measures deviation from optimal stationary feedback. For \(t\le T/2\), the finite-horizon performance identity and Riccati convergence bound its expected cumulative squared norm by \(C(R_T(s)+1)\), and state energy by \(C(t+R_T(s)+1)\) (\cref{lem:prefix-energy}, Appendix~\ref{app:prefix-energy}).

Let \(I_t^\pi(s;v)\) denote the Fisher information for \(s\) in \((x_0,u_0,\ldots,u_{t-1},x_t)\), and let \(J(\lambda)\triangleq \int |\lambda'(s)|^2/\lambda(s)\,ds\) be the information of a scalar prior density \(\lambda\), with integration over \(\{\lambda>0\}\). Appendix~\ref{app:info} derives the trajectory information under this finite-cost condition.

\begin{lemma}\label{lem:information-growth}
Suppose \cref{ass:noise-moments,ass:regular,ass:location-noise} hold. For all sufficiently large \(T\), every policy considered above, every unit direction \(v\), every \(|s|\le aT^{-1/4}\), and every integer \(0\le t\le T/2\),
\begin{equation}
 I_t^\pi(s;v)\le C_I(t+R_T(s)+1).                            \label{eq:generic-info}
\end{equation}
Here \(C_I>0\) is independent of \(\pi,v,s,t,T\). If, in addition, \(v\in\ker\cZ_{\theta_0}\), then at \(t=\lfloor T/2\rfloor\),
\begin{equation}
 I_t^\pi(s;v)\le C_I(R_T(s)+\sqrt T+1).                     \label{eq:null-info}
\end{equation}
\end{lemma}

\begin{proof}
The exact information identity in Appendix~\ref{app:info} and \eqref{eq:energy-prefix} give \eqref{eq:generic-info}. For \(v\in\ker\cZ_{\theta_0}\), affineness gives \(A[v]=B[v]K_0\), and the displacement entering the score splits as
\begin{equation}
 \begin{aligned}
 A[v]x_k+B[v]u_k
 &=B[v](u_k+K_sx_k)\\
 &\quad+B[v](K_0-K_s)x_k.
 \end{aligned}                                               \label{eq:null-score-split}
\end{equation}
The first term is \(B[v]\) times the deviation from optimal stationary feedback, so \eqref{eq:residual-prefix} bounds its cumulative expected squared norm by \(C(R_T(s)+1)\). For \(t\le T/2\), the gain bound \(\norm{K_0-K_s}=O(T^{-1/4})\) and \eqref{eq:energy-prefix} bound the corresponding quantity for the second term by \(CT^{-1/2}(t+R_T(s)+1)\). Combining these estimates in the information identity \eqref{eq:direction-FI} and using \(T^{-1/2}\le1\) gives \eqref{eq:null-info} at \(t=\lfloor T/2\rfloor\).
\end{proof}

For a block of integer times \(I=[\ell,b)\subset[1,T/2]\), let \(n=b-\ell\),
\begin{equation}
 V_I\triangleq \sum_{t\in I}x_tx_t^\top,
 \qquad
 \widetilde K_I\triangleq -\Big(\sum_{t\in I}u_tx_t^\top\Big)V_I^\dagger. \label{eq:block-LS}
\end{equation}
Project \(\widetilde K_I\) onto a fixed compact gain ball \(\cK\) containing all local \(K_s\), and call the result \(\widehat K_I\).

For sufficiently long blocks with \(b\le2n\) and \(\E_{s\sim\lambda}R_T(s)\le n\), \cref{lem:block-risk} in Appendix~\ref{app:blockgram} gives
\begin{align*}
 &\E_{s\sim\lambda}\E_s\sum_{t\in I}\norm{u_t+K_sx_t}^2\\
 &\qquad\ge c_B n\,\E_{s\sim\lambda}\E_s\fnorm{\widehat K_I-K_s}^2
 -\frac{C_B}{n}.
\end{align*}

Appendix~\ref{app:info} proves the one-dimensional van Trees inequality \cite{vantrees2001,gill1995applications} and constructs the smooth prior \(\lambda_T\) in \eqref{eq:scaled-prior} on the local line, with prior information \(O(\sqrt T)\). If \(G=DK_{\theta_0}[v]\neq0\), set \(g(s)=\ip{G}{K_s}_F/\fnorm{G}\). Since \(g'(0)=\fnorm{G}>0\), continuity gives \(g'(s)\ge\fnorm{G}/2\) throughout the prior support for sufficiently large \(T\). Applying van Trees to \(g\) and the estimator \(\ip{G}{\widehat K_I}_F/\fnorm{G}\) gives, for constants \(c,C>0\),
\begin{equation}
 \E_{s\sim\lambda_T}\E_s\fnorm{\widehat K_I-K_s}^2
 \ge \frac{c}{\E_{s\sim\lambda_T} I_b^\pi(s;v)+C\sqrt T}. \label{eq:controller-vt}
\end{equation}

\subsection{Regret lower bounds and phase classification}\label{sec:lower-proof}

\begin{proof}[Proof of \cref{thm:log-lower}]
For part~1, choose a unit direction with \(DK_{\theta_0}[v]\neq0\). If \(B_T(\lambda_T)\ge\sqrt T\), the logarithmic lower bound follows immediately. Otherwise, set \(n_\star\triangleq 2^{\lceil\log_2\sqrt T\rceil}\) and take disjoint blocks \(I_j=[n_j,2n_j)\), where \(n_j=2^jn_\star\), for integers \(j\ge0\) with \(2n_j\le T/2\). Since \(\sqrt T\le n_\star<2\sqrt T\), every block has length \(n_j>B_T(\lambda_T)\), and there are at least \(c\log T\) blocks for sufficiently large \(T\). Averaging \eqref{eq:generic-info} over \(\lambda_T\) bounds the trajectory information at time \(2n_j\) by \(C(n_j+B_T(\lambda_T)+1)=O(n_j)\). The prior information is also \(O(\sqrt T)=O(n_j)\), so \eqref{eq:controller-vt} bounds the average estimation risk below by \(c/n_j\).

For sufficiently large \(T\), all blocks exceed the fixed length threshold in \cref{lem:block-risk}, which gives average control loss at least \(c_1-C_B/n_j\ge c_1/2\) per block, with \(c_1>0\) independent of the block and policy. Summing over the blocks and applying \cref{lem:prefix-energy} yields \(B_T(\lambda_T)\ge c\log T-C\). Worst-case regret is at least this prior average.

For part~2, use the prior \eqref{eq:scaled-prior} on the line in \eqref{eq:harmful-direction}. Take the block \(I=[\lfloor T/4\rfloor,\lfloor T/2\rfloor)\), whose length is \(n\asymp T\). If \(B_T(\lambda_T)\ge n\), the result is immediate. Otherwise \cref{lem:block-risk} applies, and \cref{lem:information-growth} gives
\begin{equation*}
 \E_{s\sim\lambda_T} I_{\lfloor T/2\rfloor}^\pi(s;v)+J(\lambda_T)
 \le C(B_T(\lambda_T)+\sqrt T+1).
\end{equation*}
Combining \eqref{eq:controller-vt}, \cref{lem:block-risk}, and \cref{lem:prefix-energy} yields
\begin{equation}
 B_T(\lambda_T)+C
 \ge \frac{cT}{B_T(\lambda_T)+\sqrt T+C}.                         \label{eq:self-bound-hard}
\end{equation}
Solving the quadratic inequality \eqref{eq:self-bound-hard} gives \(B_T(\lambda_T)\ge c_0\sqrt T-C\) for some \(c_0>0\). Worst-case regret is at least this prior average. In both cases, the constants are independent of \(\pi\), so taking the infimum over all admissible policies and absorbing the additive constants for sufficiently large \(T\) proves \eqref{eq:log-lower} and \eqref{eq:sqrt-lower}.
\end{proof}

\begin{proof}[Proof of \cref{thm:phase}]
The bounded regret case follows from the nominal controller bound in \cref{prop:nominal-upper} of Appendix~\ref{app:nominal-bounds}. Under CI and \(DK_{\theta_0}\neq0\), combine \cref{thm:log-upper,thm:log-lower}. If CI fails, \eqref{eq:harmful-direction} holds for some \(v\). Combine \cref{thm:sqrt-lower} with the nominal \(O(\sqrt T)\) upper bound in \cref{prop:nominal-upper}.
\end{proof}

\section{Illustrations}\label{sec:illustrations}

\subsection{Harmless and harmful invisible directions}\label{sec:invisible-examples}
We compare two perturbations invisible under fixed nominal feedback \(u_t=-K_0x_t\): both preserve the closed-loop matrix and trajectory law, but only one preserves the optimal gain. Consider the nominal system and costs
\begin{align}
 A_0&=\diag(0.6,0),\qquad B_0=\begin{bmatrix}1\\0\end{bmatrix},\notag\\
 Q&=\diag(0.82,1),\qquad R=1.                             \label{eq:example-QR}
\end{align}
The corresponding Riccati solution, optimal gain, and closed-loop matrix are
\begin{equation*}
 P_0=I_2,\quad K_0=\begin{bmatrix}0.3&0\end{bmatrix},\quad
 F_0=\diag(0.3,0).
\end{equation*}
Here \(S_0=R+B_0^\top P_0B_0=2\). We vary \(A_s=A_0+sE\) and \(B_s=B_0+sD\), where \(s=0\) is nominal and \(E,D\) specify the changes in \(A,B\), respectively.

For the harmless direction, let
\begin{equation*}
 D_n=\begin{bmatrix}0\\1\end{bmatrix},
 \qquad E_n=D_nK_0=\begin{bmatrix}0&0\\0.3&0\end{bmatrix},
\end{equation*}
Then
\begin{equation*}
 E_n-D_nK_0=0,
 \qquad D_n^\top P_0F_0=0.
\end{equation*}
Thus the scalar parameter direction satisfies \(\cZ_0 1=0\), \(DK_0[1]=0\), and \cref{thm:exact-fiber} gives
\begin{equation*}
 P_s=P_0,\qquad K_s=K_0,
 \qquad A_s-B_sK_0=F_0
\end{equation*}
for every \(s\in\R\). Thus CI holds at \(s=0\), despite changes in the plant.

For the harmful direction, let
\begin{equation*}
 D_h=\begin{bmatrix}1\\0\end{bmatrix},
 \qquad E_h=D_hK_0=\begin{bmatrix}0.3&0\\0&0\end{bmatrix}.
\end{equation*}
Now take \(A_s=A_0+sE_h\) and \(B_s=B_0+sD_h\). Again \(E_h-D_hK_0=0\), so the closed-loop matrix under \(K_0\) remains \(F_0\). However, the optimal gain changes with \(s\): \cref{thm:directional} gives
\begin{equation*}
 DK_0[1]
 =S_0^{-1}D_h^\top P_0F_0
 =\begin{bmatrix}0.15&0\end{bmatrix}\neq0.
\end{equation*}
Thus CI fails at \(s=0\). Under \cref{ass:location-noise}, this family has local minimax regret of order \(\sqrt T\). The left panel of Figure~\ref{fig:joint-geometry} compares the optimal gains along these two invisible directions.

\subsection{A family with logarithmic local minimax regret}\label{sec:nonflat-example}
The harmless family above has a constant optimal gain. To obtain a CI example with a nonzero gain derivative at the nominal point, add a visible direction that changes only \(A\),
\begin{equation*}
 E_m=\begin{bmatrix}1&0\\0.5&0\end{bmatrix}.
\end{equation*}
Let \(\mu\) measure the added variation and \(\nu\) the original variation. The resulting family is
\begin{equation*}
 A_{\mu,\nu}=A_0+\mu E_m+\nu E_n,
 \qquad B_{\mu,\nu}=B_0+\nu D_n,
\end{equation*}
Write \(m=(1,0)^\top\) and \(n=(0,1)^\top\) for the parameter directions. We have \(\ker\cZ_0=\operatorname{span}\{n\}\) and \(DK_0[n]=0\), while Appendix~\ref{app:example} gives \(DK_0[m]\neq0\). Thus CI holds at \((0,0)\), where the gain derivative is nonzero, with local minimax rate \(\Theta(\log T)\) under \cref{ass:location-noise}.

Along \(\mu=0\), varying \(\nu\) leaves the optimal gain unchanged. When \(\mu\neq0\), however, ignoring \(\nu\) incurs a gain error of order \(|\mu\nu|\), matching the bound in \cref{thm:cross-fiber}:
\begin{equation}
 \norm{K_{\mu,\nu}-K_{\mu,0}}=\Theta(|\mu\nu|)             \label{eq:example-cross}
\end{equation}
near the origin. The upper bound follows from \cref{thm:cross-fiber}, and the lower bound from the nonzero mixed derivative in Appendix~\ref{app:example}. Figure~\ref{fig:joint-geometry} illustrates this interaction and the directional sensitivities.

\begin{figure}[htbp]
 \centering
 \includegraphics[width=\textwidth]{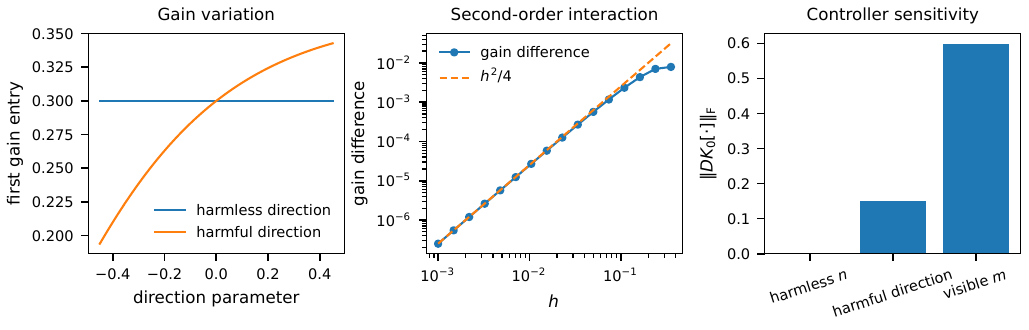}
 \caption{Optimal gain for perturbations of the same nominal plant. Left: the first gain entry along the harmless and harmful invisible directions. Middle: \(\norm{K_{h,h}-K_{h,0}}\) in the two-parameter family, compared with its leading term \(h^2/4\) as \(h\to0\). Right: norms of the nominal gain derivatives for the harmless invisible, harmful invisible, and visible directions.}
 \label{fig:joint-geometry}
\end{figure}

\subsection{A memoryless auxiliary subsystem}\label{sec:auxiliary-example}
Here the optimal controller depends on an unknown parameter \(a\) but not on another unknown parameter \(b\). Write \(x_t=(z_t,y_t)^\top\), with superscripts \((1),(2)\) denoting input and noise components, and consider
\begin{align*}
 z_{t+1}&=az_t+u_t^{(1)}+w_{t+1}^{(1)},\\
 y_{t+1}&=bu_t^{(2)}+w_{t+1}^{(2)},\qquad Q=R=I_2.        
\end{align*}
The first channel is a scalar LQR with an unknown state-transition coefficient. The second has no state memory: its current state does not affect its next state, and a nonzero input only increases expected cost. In matrix form, \(A=\diag(a,0)\) and \(B=\diag(1,b)\). Solving the DARE for the two decoupled subsystems gives
\begin{equation*}
 P=\diag(p(a),1),\qquad K=\diag(k(a),0),
\end{equation*}
where \(p(a)=(a^2+\sqrt{a^4+4})/2\) and \(k(a)=ap(a)/(1+p(a))\), with \(k'(a)>0\). Under any fixed nominal optimal feedback, changes in \(a\) change the first channel's closed-loop dynamics, whereas changes in \(b\) are invisible because \(u_t^{(2)}=0\). Changes in \(b\) also leave the optimal gain unchanged. Thus CI holds at every parameter, although the plant is not identifiable under optimal feedback.

Use the FCE variant with fixed projection radius in Remark~\ref{rem:radius-policy} to estimate only \(a\). Its controllers set \(u_t^{(2)}=0\), so both the error term in \eqref{eq:nuisance-identity} and the gain error from ignoring \(b\) vanish. The proof of \eqref{eq:radius-CI} gives \(C[1+\log(1+Tr^2)]\) regret without the \(Tr^4\) term, with a matching lower bound under \cref{ass:location-noise}. Thus fixed sufficiently small neighborhoods of a known nominal model admit logarithmic regret even though \(A\) is singular.

\subsection{Numerical illustration of regret rates}
We use the families in Sections~\ref{sec:invisible-examples} and \ref{sec:nonflat-example} to illustrate the predicted loss growth and the distinction between parameter and controller error. The simulations use the costs in \eqref{eq:example-QR}, \(x_0=0\), and Gaussian process noise with \(W=0.2I_2\). For \(T=2^9,\ldots,2^{14}\), we take \((\mu,\nu)=(0.8,0.6)T^{-1/4}\) in the CI family and \(s=T^{-1/4}\) in the harmful family. FCE uses \(N_0=32\), \(c_\rho=1.2\), \(\bar\rho=0.45\), and the Frobenius ball centered at \(F_0\) with radius \(0.75\) for matrix projection.

The plotted performance measure is the cumulative stationary Bellman loss \(\sum_{t<T}\norm{u_t+K_\theta x_t}_{S_\theta}^2\), averaged over runs. The stationary Bellman identity relates its expectation to average-cost regret through the boundary term \(x_0^\top P_\theta x_0-\E_\theta^\pi[x_T^\top P_\theta x_T]\). Under the uniform state moment bound \eqref{eq:uniform-state}, this term is bounded. The comparison with finite-horizon regret then follows from \cref{lem:benchmark}.

Figure~\ref{fig:phase-scaling} shows mean Bellman loss approximately affine in \(\log T\) for FCE in the CI family and in \(\sqrt T\) for nominal feedback in the harmful family. Nominal feedback has zero Bellman loss along the harmless direction because \(K_s=K_0\).

\begin{figure}[htbp]
 \centering
 \includegraphics[width=\textwidth]{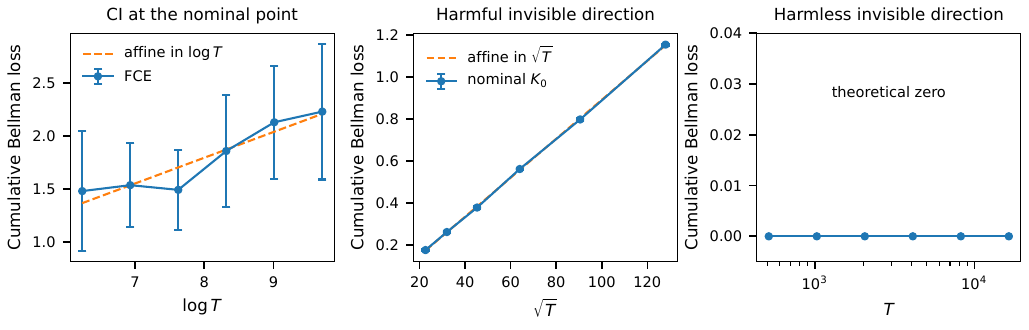}
 \caption{Cumulative stationary Bellman loss in the three local regimes under \(T^{-1/4}\) perturbations. The first two panels show means and twice the Monte Carlo standard error over 20 runs. Dashed lines are least-squares affine fits to the plotted means. The third shows the theoretical zero loss along the harmless invisible direction.}
 \label{fig:phase-scaling}
\end{figure}

In Figure~\ref{fig:controller-plant}, the full parameter error approaches the magnitude of the invisible coordinate, which FCE does not estimate. The errors in the visible coordinate and the controller nevertheless decrease.

\begin{figure}[htbp]
 \centering
 \includegraphics[width=0.72\textwidth]{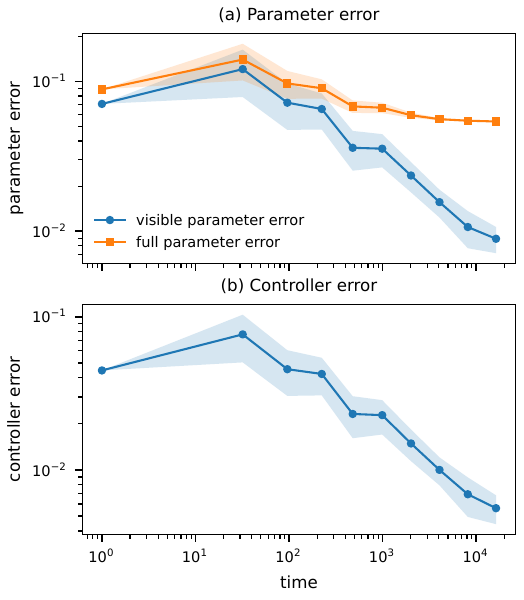}
 \caption{Errors for the family in Section~\ref{sec:nonflat-example}, with \((\mu,\nu)=(0.8,0.6)T^{-1/4}\) and \(T=16384\), averaged over 30 runs. The visible parameter, full parameter, and controller errors are \(|\widehat\mu_j-\mu|\), \(\sqrt{(\widehat\mu_j-\mu)^2+\nu^2}\), and \(\fnorm{K_{\widehat\mu_j,0}-K_{\mu,\nu}}\), respectively. Shaded bands are twice the Monte Carlo standard error.}
 \label{fig:controller-plant}
\end{figure}

\FloatBarrier
\section{Conclusion}\label{sec:conclusion}
This paper characterized local minimax regret in structured adaptive linear--quadratic regulation through controller identifiability. For affine models with jointly unknown state and input matrices, this condition separates logarithmic from square-root regret in the shrinking neighborhoods of a known nominal model considered here, under the stated regularity assumptions and a nonzero nominal gain derivative. A vanishing derivative permits bounded regret. The characterization identifies when model differences invisible under nominal optimal feedback matter for control. Under controller identifiability, sufficiently small invisible perturbations leave the optimal gain unchanged, allowing the proposed certainty-equivalent policy to attain logarithmic regret by estimating only the visible parameter component.

\appendix

\section{Riccati Equations and Controller Identifiability}\label{app:riccati}
\subsection{Riccati derivatives and invariance}
Derivatives of \(P_\theta\) and \(F_\theta\) use the same directional convention as \(DK_\theta\).
\begin{lemma}\label{thm:directional}
Fix an interior parameter \(\theta\) with \((A_\theta,B_\theta)\) stabilizable and \((Q^{1/2},A_\theta)\) observable, and \(v\in\R^d\). Let \(Z=\cZ_\theta v\).
Then
\begin{align}
 DP_\theta[v]
 &=F_\theta^\top DP_\theta[v]F_\theta
   +Z^\top P_\theta F_\theta+F_\theta^\top P_\theta Z,       \label{eq:DP}\\
 S_\theta DK_\theta[v]
 &=B[v]^\top P_\theta F_\theta+B_\theta^\top DP_\theta[v]F_\theta
   +B_\theta^\top P_\theta Z.                               \label{eq:DK}
\end{align}
If \(v\in\ker\cZ_\theta\), then
\begin{equation}
 DP_\theta[v]=0,
 \qquad
 S_\theta DK_\theta[v]=B[v]^\top P_\theta F_\theta.          \label{eq:null-derivatives}
\end{equation}
\end{lemma}

\begin{proof}
The smoothness established in \cref{lem:benchmark} permits differentiation of the Bellman identity \eqref{eq:bellman-riccati} and the stationarity identity
\begin{equation*}
 RK_\theta=B_\theta^\top P_\theta F_\theta.
\end{equation*}
Since
\(
DF_\theta[v]=Z-B_\theta DK_\theta[v]
\), differentiating the Bellman identity and canceling the terms involving \(DK_\theta[v]\) by stationarity gives \eqref{eq:DP}. Differentiating the stationarity identity gives \eqref{eq:DK}. If \(Z=0\), then \eqref{eq:DP} reduces to \(DP_\theta[v]=F_\theta^\top DP_\theta[v]F_\theta\), and Schur stability implies \(DP_\theta[v]=0\).
\end{proof}

\begin{lemma}\label{thm:exact-fiber}
Fix an interior parameter \(\theta\) with \((A_\theta,B_\theta)\) stabilizable and \((Q^{1/2},A_\theta)\) observable. If \(v\in\ker\cZ_\theta\) and \(DK_\theta[v]=0\), then for every scalar \(s\) such that \(\theta+sv\in\Theta\) and \((Q^{1/2},A_{\theta+sv})\) is observable,
\begin{equation*}
 P_{\theta+sv}=P_\theta,\quad
 K_{\theta+sv}=K_\theta,\quad
 F_{\theta+sv}=F_\theta.
\end{equation*}
\end{lemma}

\begin{proof}
Affineness and \(A[v]-B[v]K_\theta=0\) imply
\(
A_{\theta+sv}-B_{\theta+sv}K_\theta=F_\theta
\)
exactly.  By \cref{thm:directional}, \(DK_\theta[v]=0\) also gives \(B[v]^\top P_\theta F_\theta=0\), so
\(
B_{\theta+sv}^\top P_\theta F_\theta=RK_\theta
\).
The same \((P_\theta,K_\theta)\) satisfies the Bellman and stationarity identities at \(\theta+sv\). The unchanged Schur stable closed-loop matrix ensures stabilizability of the system at this parameter. Observability at the target and uniqueness in \cref{lem:benchmark} identify this pair with the optimal solution there.
\end{proof}

\subsection{Proofs of the structural results}\label{app:quotient-proof}\label{app:physical-proof}\label{app:cross-fiber-proof}
\begin{proof}[Proof of \cref{thm:quotient}]
Use the spaces defined in \eqref{eq:split-theta}.
Choose a sufficiently small ball about \(\theta\) where stabilizability and observability hold. Under CI at \(\theta\), \cref{thm:exact-fiber} gives the second statement. Conversely, for every \(n\in\ker\cZ_\theta\), the second statement gives \(K_{\theta+sn}=K_\theta\) for all sufficiently small \(s\), so
\begin{equation*}
 DK_\theta[n]=\left.\frac{d}{ds}K_{\theta+sn}\right|_{s=0}=0.
\end{equation*}
Thus CI holds, proving equivalence of the first two statements.

Suppose the first statement holds. Write \(\theta'-\theta=n+m\) with \(n\in\cN_\theta\), \(m\in\cM_\theta\). Orthogonality gives \(\norm{n}\le\norm{\theta'-\theta}\), so \(\theta+n\) remains in the chosen ball. By \cref{thm:exact-fiber}, \(K_{\theta+n}=K_\theta\). Injectivity of \(\cZ_\theta|_{\cM_\theta}\) in finite dimensions gives \(\norm{m}\le c\fnorm{\cZ_\theta m}\) for some \(c>0\) (trivially if \(\cM_\theta=\{0\}\)). Local smoothness on a sufficiently small closed ball therefore gives constants \(L,C>0\) such that
\begin{align*}
 \norm{K_{\theta'}-K_\theta}
 &=\norm{K_{\theta+n+m}-K_{\theta+n}}\\
 &\le L\norm{m}\le C\fnorm{\cZ_\theta m}.
\end{align*}
 Since \(\cZ_\theta m=\cZ_\theta(\theta'-\theta)\), this proves the third statement.

Conversely, the third statement gives \(K_{\theta+sn}=K_\theta\) for \(n\in\cN_\theta\) and small \(s\). Differentiating gives CI at \(\theta\).
\end{proof}

\begin{proof}[Proof of \cref{cor:physical-identifiability}]
By \eqref{eq:null-derivatives}, every \(v\in\cN_\theta\) satisfies \(B[v]^\top P_\theta F_\theta=0\) under CI at \(\theta\). The map \(v\mapsto B[v]\) on \(\cN_\theta\) has kernel \(\mathcal D\), since \(A[v]=B[v]K_\theta\), and each column of \(B[v]\) lies in \(\ker(F_\theta^\top P_\theta)\). Rank--nullity and \(P_\theta\succ0\) give
\begin{align*}
 \dim\cN_\theta-\dim\mathcal D
 &\le n_u\dim\ker(F_\theta^\top P_\theta)\\
 &=n_u(n_x-\rank F_\theta).
\end{align*}
Stationarity gives
\begin{equation*}
 A_\theta=(I+B_\theta R^{-1}B_\theta^\top P_\theta)F_\theta.
\end{equation*}
The prefactor is similar to \(I+P_\theta^{1/2}B_\theta R^{-1}B_\theta^\top P_\theta^{1/2}\succ0\), hence invertible, so \(\rank F_\theta=\rank A_\theta\). This proves \eqref{eq:physical-nullity}. If \(A_\theta\) is nonsingular, the right-hand side is zero, and \(\mathcal D\subseteq\cN_\theta\) gives \(\cN_\theta=\mathcal D\). Conversely, directions in \(\mathcal D\) leave the gain unchanged, so \(\cN_\theta=\mathcal D\) implies CI.
\end{proof}

\begin{proof}[Proof of \cref{thm:cross-fiber}]
We write \(D^2K_\theta[v,z]\triangleq \left.\frac{d}{ds}DK_{\theta+sz}[v]\right|_{s=0}\) for the second derivative in directions \(v,z\in\R^d\), and likewise for \(P_\theta\).
Choose a closed ball of radius \(r_2>0\) about \(\theta_0\) within the region of stabilizability and observability, on which both second derivatives are bounded. If \(\norm{\mu}+\norm{\nu}\le r_2\), then \(\theta_0+s\mu+t\nu\) lies in this ball for all \(s,t\in[0,1]\). Applying the fundamental theorem of calculus twice yields
\begin{align*}
 &K_{\theta_0+\mu+\nu}-K_{\theta_0+\mu}-K_{\theta_0+\nu}+K_{\theta_0}\\
 &\qquad=\int_0^1\!\int_0^1 D^2K_{\theta_0+s\mu+t\nu}[\mu,\nu]\,dt\,ds.
\end{align*}
By \cref{thm:exact-fiber}, \(K_{\theta_0+\nu}=K_{\theta_0}\), so the last two terms on the left cancel. Bounding the integral gives \eqref{eq:cross-fiber}. The same argument applies to \(P\), with \(L_2\) chosen to bound both second derivatives.
\end{proof}

\subsection{Benchmark equivalence and nominal control}\label{app:nominal-bounds}
\begin{lemma}\label{lem:benchmark}
For every \(\theta\in\Theta\) such that \((A_\theta,B_\theta)\) is stabilizable and \((Q^{1/2},A_\theta)\) observable, the DARE has a unique positive semidefinite solution, which is positive definite and stabilizing. The maps \(P_\theta,K_\theta\) are smooth at interior parameters satisfying these conditions.
Under \cref{ass:noise-moments,ass:regular}, there are a neighborhood \(U_0\subset\Theta\) of \(\theta_0\) on which both conditions hold and constants \(C_b<\infty\), \(0<\beta<1\), such that, uniformly over \(\theta\in U_0\) and integers \(h,T\ge1\),
\begin{align}
 \norm{K_{\theta,h}-K_\theta}+\norm{P_{\theta,h}-P_\theta}
 &\le C_b\beta^h,                                             \label{eq:riccati-exp}\\
 \left|V_T^\star(\theta)-TJ_\theta\right|&\le C_b.           \label{eq:benchmark-equivalence}
\end{align}
Consequently \(|R_T^\pi(\theta)-\bar R_T^\pi(\theta)|\le C_b\) for every \(\theta\in U_0\) and policy with finite expected cumulative cost.
\end{lemma}

\begin{proof}
For the finite-horizon problem, set $P_{\theta,0}=0$ and for $h\ge1$ define
\begin{align}
 S_{\theta,h}&=R+B_\theta^\top P_{\theta,h-1}B_\theta,\notag\\
 K_{\theta,h}&=S_{\theta,h}^{-1}B_\theta^\top P_{\theta,h-1}A_\theta,\label{eq:app-fh-rec}\\
 P_{\theta,h}&=Q+K_{\theta,h}^\top RK_{\theta,h}\notag\\
 &\quad +(A_\theta-B_\theta K_{\theta,h})^\top P_{\theta,h-1}
 (A_\theta-B_\theta K_{\theta,h}).\notag
\end{align}
For each parameter satisfying the stated conditions, $P_{\theta,h}\succeq0$ increases with $h$ and is bounded above by the cost matrix of a stabilizing feedback. Since $S_{\theta,h}\succeq R\succ0$, its limit solves the DARE. For any positive semidefinite solution, define $K_\theta,F_\theta$ by \eqref{eq:optimal-quantities}. If $F_\theta z=\lambda z$ for a nonzero complex vector $z$ and $|\lambda|\ge1$, \eqref{eq:bellman-riccati} gives
\begin{equation*}
 (1-|\lambda|^2)z^*P_\theta z
 =\norm{Q^{1/2}z}^2+\norm{R^{1/2}K_\theta z}^2.
\end{equation*}
Both sides vanish, so $Q^{1/2}z=K_\theta z=0$ and $A_\theta z=\lambda z$, contradicting observability. Thus $F_\theta$ is Schur stable.

If $x^\top P_\theta x=0$, iterating \eqref{eq:bellman-riccati} gives $Q^{1/2}F_\theta^jx=K_\theta F_\theta^jx=0$ for all $j\ge0$. Hence $F_\theta^jx=A_\theta^jx$, and observability gives $x=0$, proving $P_\theta\succ0$. Completing the square in the Bellman identity under each solution's stable feedback gives opposite inequalities between the two cost matrices, proving uniqueness \cite{anderson2007optimal,lancaster1995algebraic,hager1976convergence}.

At such an interior parameter, stationarity gives the derivative $X\mapsto X-F_\theta^\top XF_\theta$ of the DARE residual with respect to $P$. Schur stability makes this map invertible on symmetric matrices, so the smooth implicit function theorem and \eqref{eq:optimal-quantities} give smoothness of $P_\theta,K_\theta$.

Stabilizability and observability persist on a sufficiently small neighborhood of $\theta_0$. On a smaller compact neighborhood, choose $H\succ0$ and $0<q<1$ with $F_\theta^\top HF_\theta\preceq q^2H$. In the norm $\norm{X}_{H,*}\triangleq \norm{H^{-1/2}XH^{-1/2}}$, the derivative of the Riccati iteration at $P_\theta$ is $X\mapsto F_\theta^\top XF_\theta$ and has operator norm at most $q^2$. Continuity and compactness give balls of a common radius about the solutions with a common contraction factor $\beta\in(q^2,1)$.

Since the continuous functions $\lambda_{\max}(P_\theta-P_{\theta,h})$ decrease to zero, Dini's theorem gives an index $h_0$, uniform in $\theta$, after which the iterates lie in these contraction neighborhoods. Thus, uniformly in $\theta$, for $h\ge h_0$,
\begin{equation*}
 \norm{P_{\theta,h}-P_\theta}_{H,*}
 \le\beta^{h-h_0}\norm{P_{\theta,h_0}-P_\theta}_{H,*}.
\end{equation*}
The gain formula \eqref{eq:app-fh-rec} is uniformly Lipschitz in $P_{\theta,h-1}$ on this region, so \eqref{eq:riccati-exp} follows by norm equivalence, with the earlier iterates absorbed into $C_b$.

For the fixed initial state, dynamic programming gives
\begin{equation*}
 V_T^\star(\theta)=x_0^\top P_{\theta,T}x_0
 +\sum_{h=1}^T\tr(P_{\theta,h-1}W).
\end{equation*}
Uniform boundedness of \(P_{\theta,T}\) and the geometric sum in \eqref{eq:riccati-exp} give \eqref{eq:benchmark-equivalence}, enlarging \(C_b\) if needed. Finite policy cost, \(R\succ0\), and the dynamics with finite noise variance ensure finite control and state second moments over the horizon. Completing the square with $P_{\theta,T-t-1}$ in the dynamic programming identity and telescoping gives \eqref{eq:FH-PDI}, with zero terminal value since $P_{\theta,0}=0$.
\end{proof}

\begin{proposition}\label{prop:nominal-upper}
Under \cref{ass:noise-moments,ass:regular}, fix \(\varepsilon>0\). There are \(T_0<\infty\) and \(C<\infty\) such that, for every \(T\ge T_0\), the stationary nominal controller \(u=-K_0x\) satisfies, uniformly over \(\theta\in\Theta_T(\theta_0,\varepsilon)\),
\begin{equation}
 R_T^{K_0}(\theta)\le C T\norm{K_\theta-K_0}^2+C.             \label{eq:nominal-upper-general}
\end{equation}
Consequently,
\begin{align}
 R_T^{K_0}(\theta)&\le C\sqrt T, \label{eq:nominal-sqrt}\\
 R_T^{K_0}(\theta)&\le C \quad\text{if }DK_{\theta_0}=0.   \label{eq:nominal-flat}
\end{align}
\end{proposition}

\begin{proof}
Choose a nominal neighborhood on which \(A_\theta-B_\theta K_0\) has a common quadratic Lyapunov function, and then choose \(T_0\) so that \(\Theta_T(\theta_0,\varepsilon)\) lies in this neighborhood for every \(T\ge T_0\). Let \(u_t=-K_0x_t\). Then \(\sup_t\E_\theta\norm{x_t}^2\le C_x\) uniformly on that neighborhood. The stationary Bellman identity at the true parameter gives
\begin{align*}
 V_T^{K_0}(\theta)-TJ_\theta
 &=x_0^\top P_\theta x_0-\E_\theta[x_T^\top P_\theta x_T]\\
 &\quad+\E_\theta\sum_{t=0}^{T-1}\norm{(K_\theta-K_0)x_t}_{S_\theta}^2.
\end{align*}
The initial term is bounded, the terminal term is nonpositive, and \(S_\theta\) is uniformly bounded. Combining the resulting upper bound with \cref{lem:benchmark} proves \eqref{eq:nominal-upper-general}. Taylor's theorem gives \(\norm{K_\theta-K_0}=O(T^{-1/4})\), or \(O(T^{-1/2})\) if \(DK_{\theta_0}=0\). Substitution into \eqref{eq:nominal-upper-general} yields \eqref{eq:nominal-sqrt} and \eqref{eq:nominal-flat}, respectively.
\end{proof}

\subsection{Mixed derivative of the optimal gain}\label{app:example}
Consider the family in Section~\ref{sec:nonflat-example}, with parameter directions \(m,n\) defined there. We use the second derivative convention defined in the proof of \cref{thm:cross-fiber}.
Applying \cref{thm:directional} in direction \(m\) gives
\begin{equation*}
 DP_0[m]=\diag(60/91,0),\qquad
 DK_0[m]=\begin{bmatrix}109/182&0\end{bmatrix}.
\end{equation*}
Differentiate the Lyapunov equation \eqref{eq:DP} for direction \(n\) along \(m\). At the origin, \(\cZ_0n=0\) and
\(
\left.\frac{d}{ds}(\cZ_{sm}n)\right|_{s=0}=-D_nDK_0[m]
\).
Here
\(
-D_nDK_0[m]=\begin{bmatrix}0&0\\-109/182&0\end{bmatrix}
\),
so both forcing terms in the differentiated Lyapunov equation vanish and \(D^2P_0[m,n]=0\). In the derivative of \eqref{eq:DK}, all remaining terms vanish except
\(
D_n^\top P_0(E_m-B_0DK_0[m])=\begin{bmatrix}1/2&0\end{bmatrix}
\).
Since \(S_0=2\),
\begin{equation}
 D^2K_0[m,n]=\begin{bmatrix}1/4&0\end{bmatrix}.             \label{eq:example-mixed-derivative}
\end{equation}
The identity \(K_{0,\nu}=K_{0,0}\) and two applications of the fundamental theorem of calculus give
\begin{equation*}
 K_{\mu,\nu}-K_{\mu,0}
 =\mu\nu\int_0^1\!\int_0^1
 D^2K_{s\mu,t\nu}[m,n]\,ds\,dt.
\end{equation*}
Continuity and the nonzero first component in \eqref{eq:example-mixed-derivative} give matching upper and lower bounds for all sufficiently small \(\mu,\nu\), proving \eqref{eq:example-cross}.

\section{Estimation and Stability under FCE}\label{app:FCE}
\subsection{Restricted inversion and local stability}
For \(\cM_{\theta_0}\ne\{0\}\), let \(\sigma_0>0\) be the smallest singular value of \(\cL_{K_0}\).
\begin{lemma}\label{lem:left-inverse}
There is \(r_L>0\) such that \(\norm{K-K_0}\le r_L\) implies
\begin{equation*}
 \sigma_{\min}(\cL_K)\ge\frac{\sigma_0}{2},
 \qquad
 \opnorm{\cL_K^\dagger}\le\frac{2}{\sigma_0}.
\end{equation*}
\end{lemma}

\begin{proof}
For \(m\in\cM_{\theta_0}\), \((\cL_K-\cL_{K_0})m=-B[m](K-K_0)\). Boundedness of the linear map \(m\mapsto B[m]\) gives \(\opnorm{\cL_K-\cL_{K_0}}\to0\) as \(K\to K_0\). Continuity of singular values proves the bounds.
\end{proof}

\begin{lemma}\label{lem:nuisance}
Under \cref{ass:regular}, there is \(C_\nu\) such that, for the FCE updates with sufficiently small maximum projection radius \(\bar\rho\),
\begin{equation*}
 \fnorm{\cZ_{K_j}\nu}\le C_\nu\norm{\nu}\rho_j.
\end{equation*}
\end{lemma}

\begin{proof}
Combine \eqref{eq:nuisance-identity}, boundedness of \(B[\cdot]\), and local Lipschitz continuity of \(K_{\theta_0+\eta}\).
\end{proof}

\begin{lemma}\label{lem:common-lyap}
Write \(G(\theta,\eta)\triangleq A_\theta-B_\theta K_{\theta_0+\eta}\).
Under \cref{ass:regular}, for all sufficiently small \(\bar\rho>0\), there are \(H\succ0\) and \(\alpha>0\) such that
\begin{equation}
 G(\theta,\eta)^\top HG(\theta,\eta)\preceq H-\alpha I       \label{eq:common-lyap}
\end{equation}
for every \(\theta\in\Theta\) with \(\norm{\theta-\theta_0}\le\bar\rho\) and every \(\eta\in\cM_{\theta_0}\) with \(\norm{\eta}\le\bar\rho\).
If \cref{ass:noise-moments} also holds, there are constants \(C_x,C_H<\infty\) such that, under arbitrary switching among FCE controllers, for integers \(\tau\ge0\) and \(N\ge1\),
\begin{align}
 \sup_t\E\norm{x_t}^2&\le C_x,                              \label{eq:uniform-state}\\
 \E\!\left[\sum_{s=0}^{N-1}\norm{x_{\tau+s}}^2\mid\mathcal F_\tau\right]
 &\le C_H(N+x_\tau^\top Hx_\tau).                           \label{eq:conditional-energy}
\end{align}
\end{lemma}

\begin{proof}
Let \(H-F_{\theta_0}^\top HF_{\theta_0}=I\). Continuity gives \eqref{eq:common-lyap} on a sufficiently small product neighborhood. Conditional drift gives \(\E[x_t^\top Hx_t]\le(1-\gamma)^t x_0^\top Hx_0+\tr(HW)/\gamma\) for some \(\gamma\in(0,1)\), proving \eqref{eq:uniform-state}. Telescoping the same drift from \(\tau\) gives \eqref{eq:conditional-energy}.
\end{proof}

\subsection{Excitation from nondegenerate noise}\label{app:noise}
Positive definite covariance supplies excitation uniformly over state directions and predictable shifts.

\begin{lemma}\label{lem:location-smallball}
Under \cref{ass:noise-moments}, there are $a_0,p_0>0$, depending on the fixed noise law, such that
\begin{equation*}
 \inf_{\norm{q}=1}\inf_{a\in\R}
 \Pp\bigl(|q^\top w-a|\ge a_0\bigr)\ge p_0.
\end{equation*}
\end{lemma}

\begin{proof}
Let $w'$ be an independent copy of $w$, and put $D\triangleq w-w'$. The function $q\mapsto\E\min\{(q^\top D)^2,1\}$ is continuous by dominated convergence and positive on the unit sphere because $\E(q^\top D)^2=2q^\top Wq>0$. Compactness therefore gives
\begin{equation*}
 m\triangleq\min_{\norm{q}=1}\E\min\{(q^\top D)^2,1\}>0.
\end{equation*}
Set $b\triangleq\sqrt{m/2}$. Since $m\le b^2+\Pp(|q^\top D|\ge b)$, for every unit $q$ and every $a\in\R$,
\begin{equation*}
 m/2\le\Pp(|q^\top D|\ge b)
 \le2\Pp(|q^\top w-a|\ge b/2).
\end{equation*}
The last inequality follows from the triangle inequality and the identical laws of $w,w'$. Thus $a_0=\sqrt{m/8}$ and $p_0=m/4$ suffice.
\end{proof}

\subsection{Lower-tail bounds for Gram matrices and projected OLS}\label{app:ols}
For one epoch, use time relative to its deterministic start and write
\begin{equation*}
 x_{t+1}=Gx_t+w_{t+1}.
\end{equation*}
Here \(G\) is the closed-loop matrix held fixed during an epoch of length \(N\). Omit the first regressor, let \(V_N=\sum_{t=1}^{N-1}x_tx_t^\top\), and denote by \(\widehat G_N\) the OLS estimate in \eqref{eq:app-OLSdef}, projected onto \(\cG\).

\begin{lemma}\label{lem:OLS}
Under \cref{ass:noise-moments,ass:regular}, there is \(C_G<\infty\), depending on \(x_0\), the fixed noise law and nominal neighborhood, such that every FCE epoch of integer length \(N\ge4\) satisfies, uniformly over the local parameter class,
\begin{equation}
 \E\fnorm{\widehat G_N-G}^2\le \frac{C_G}{N}.                \label{eq:OLS-risk}
\end{equation}
The expectation also averages over the history at the epoch start.
\end{lemma}

To prove \cref{lem:OLS}, we first bound the probability that the empirical Gram matrix has a small eigenvalue. The same bound is used in Appendix~\ref{app:blockgram}.

\begin{lemma}\label{lem:app-Gram}
Let \(I=[\ell,b)\), with integers \(1\le\ell<b\), be a deterministic block of \(n=b-\ell\) states, and set \(V_I=\sum_{t\in I}x_tx_t^\top\). Suppose that, for fixed \(a_0,p_0,C>0\),
\begin{equation*}
 \Pp(|q^\top x_t|\ge a_0\mid\mathcal F_{t-1})\ge p_0,
 \qquad \E\tr(V_I)\le Cn,
\end{equation*}
for every unit \(q\) and \(t\in I\). There are \(c_g,C_g,n_g>0\), depending only on these constants and the state dimension, such that
\begin{equation}
 \Pp\!\left(V_I\nsucceq c_gnI_{n_x}\right)
 \le C_gn^{-2},\qquad n\ge n_g.                              \label{eq:app-Gram}
\end{equation}
\end{lemma}

\begin{proof}
For a fixed unit \(q\), apply Azuma--Hoeffding to the centered indicators of \(\{|q^\top x_t|\ge a_0\}\). This gives
\begin{equation}
 \Pp\!\left(q^\top V_Iq<\frac{a_0^2p_0n}{2}\right)
 \le e^{-cn}.                                               \label{eq:app-fixedq}
\end{equation}
Put \(c_1=a_0^2p_0/2\). Markov's inequality gives \(\Pp(\opnorm{V_I}>n^3)\le Cn^{-2}\). Take a \(c_1/(4n^2)\)-net of the unit sphere, of polynomial cardinality in \(n\) for fixed state dimension. By \eqref{eq:app-fixedq} and a union bound, all net directions satisfy \(q^\top V_Iq\ge c_1n\) except on an exponentially small event. When this bound and \(\opnorm{V_I}\le n^3\) hold, each unit \(u\) and its nearest net point \(q\) satisfy
\begin{equation*}
 u^\top V_Iu\ge q^\top V_Iq
 -2\norm{u-q}\opnorm{V_I}\ge c_1n/2.
\end{equation*}
Combining the exceptional probabilities proves \eqref{eq:app-Gram}. This is the small-ball argument of \cite{simchowitz2018learning}, with a trace bound supplying the upper tail.
\end{proof}

\begin{proof}[Proof of \cref{lem:OLS}]
At the deterministic epoch start \(\tau\), \(G\) is \(\mathcal F_\tau\)-measurable. All indices below are relative to \(\tau\). Fresh noise and \cref{lem:location-smallball} give the conditional probability bound in \cref{lem:app-Gram}, and \eqref{eq:uniform-state} gives \(\E\tr(V_N)\le CN\). Apply that lemma with \(n=N-1\). After adjusting constants, there is \(N_g<\infty\) such that \(\Pp(V_N\nsucceq c_gNI_{n_x})\le C_gN^{-2}\) for \(N\ge N_g\).

Let
\begin{equation}
 \widetilde G_N=\left(\sum_{t=1}^{N-1}x_{t+1}x_t^\top\right)V_N^\dagger. \label{eq:app-OLSdef}
\end{equation}
On the event \(V_N\succeq c_gNI_{n_x}\),
\begin{equation*}
 \widetilde G_N-G=S_NV_N^{-1},
 \qquad S_N\triangleq \sum_{t=1}^{N-1}w_{t+1}x_t^\top.
\end{equation*}
Since \(G\in\cG\), projection cannot increase the estimation error, so
\begin{equation*}
 \fnorm{\widehat G_N-G}^2\le \frac{\fnorm{S_N}^2}{c_g^2N^2}. 
\end{equation*}
The summands \(w_{t+1}x_t^\top\) are martingale differences, so their cross terms have zero expectation. Independence of \(w_{t+1}\) from \(x_t\) and \eqref{eq:uniform-state} then give
\begin{equation}
 \E\fnorm{S_N}^2
 =\sum_{t=1}^{N-1}\E[\norm{w_{t+1}}^2\norm{x_t}^2]\le CN.  \label{eq:app-SN}
\end{equation}
On the bad event, projection bounds the squared error by $\operatorname{diam}(\cG)^2$, where \(\operatorname{diam}\) denotes diameter in the Frobenius norm. Combining \eqref{eq:app-Gram}--\eqref{eq:app-SN} proves \eqref{eq:OLS-risk} uniformly over FCE epoch matrices for \(N\ge N_g\). The same projection bound covers the finitely many \(4\le N<N_g\) after increasing \(C_G\).
\end{proof}

\section{Information Inequalities and Regret Lower Bounds}\label{app:lower-tech}

\subsection{Fisher information and the van Trees inequality}\label{app:info}
\begin{lemma}\label{lem:sequential-DQM}
Under \cref{ass:noise-moments,ass:location-noise}, \(J_w\succ0\). Fix a finite \(t\), a scalar affine submodel \(\theta_s=\theta_0+sv\), and any causal policy represented by parameter-independent stochastic kernels. If
\(
\E_s\sum_{k<t}\norm{A[v]x_k+B[v]u_k}^2<\infty
\),
then the trajectory experiment through time \(t\) is differentiable in quadratic mean (DQM) at \(s\). Put \(d_k(v)\triangleq A[v]x_k+B[v]u_k\). Its score is
\[
 S_t(s)=\sum_{k=0}^{t-1}d_k(v)^\top\psi(w_{k+1}),
\]
and its Fisher information is
\begin{equation}
 I_t^\pi(s;v)
 =\E_s\sum_{k=0}^{t-1}
 d_k(v)^\top J_wd_k(v).                                    \label{eq:direction-FI}
\end{equation}
\end{lemma}

\begin{proof}
Since \(p=f^2\), its weak gradient is \(\nabla p=2f\nabla f=-p\psi\). Weak integration by parts against smooth truncations of \(1,w_1,\ldots,w_{n_x}\) gives \(\E\psi(w)=0\) and \(\E[w\psi(w)^\top]=I\). Indeed, \cref{ass:noise-moments} and Cauchy--Schwarz make \((1+\norm{w})\norm{\nabla p(w)}\) integrable. Cutoffs equal to one on the ball of radius \(L\), zero outside the ball of radius \(2L\), and with gradient bounded by \(C/L\) have boundary terms tending to zero as \(L\to\infty\). Therefore
\begin{equation*}
 \begin{aligned}
 &\E\!\left[(\psi(w)-W^{-1}w)(\psi(w)-W^{-1}w)^\top\right]\\
 &\qquad=J_w-W^{-1}\succeq0,
 \end{aligned}
\end{equation*}
so \(J_w\succeq W^{-1}\succ0\).

Write $A_s=A_{\theta_s}$, $B_s=B_{\theta_s}$, and $h_k=(x_0,u_0,\ldots,u_{k-1},x_k)$. Denote the policy kernels by $\pi_k(du_k\mid h_k)$, which are independent of $s$.

For a finite time $t$, choose a reference density $r$ positive everywhere on $\R^{n_x}$. With $x_0$ held fixed, the measure on the remaining trajectory coordinates
\begin{equation*}
 \mu_t^\pi(dh_t)
 \triangleq \prod_{k=0}^{t-1}
 \pi_k(du_k\mid h_k)r(x_{k+1})dx_{k+1}
\end{equation*}
depends on the fixed policy $\pi$ but not on $s$, and dominates every trajectory law $\Pp_s$ through time $t$, with density
\begin{equation}
 \ell_s(h_t)
 =\prod_{k=0}^{t-1}
 \frac{p(x_{k+1}-A_sx_k-B_su_k)}{r(x_{k+1})}.                \label{eq:app-dominated-density}
\end{equation}
For a scalar or vector function \(g\), write \(\norm{g}_2^2\triangleq \int\norm{g(w)}^2\,dw\) for its squared \(L^2\) norm. For every deterministic vector \(d\in\R^{n_x}\) and scalar parameter increment \(h\),
\begin{equation*}
 \norm{f(\cdot-hd)-f+h\,d^\top\nabla f}_2
 \le2|h|\norm{d}\norm{\nabla f}_2,
\end{equation*}
and the left side is $o(|h|)$ for fixed $d$. Conditional on the history and current input, take $d=d_k(v)$. Squaring this bound and averaging under the base law $\Pp_s$, the finite energy assumption gives
\begin{equation*}
 \E_s\norm{f(\cdot-hd_k)-f+h\,d_k^\top\nabla f}_2^2=o(h^2).
\end{equation*}
Here $d_k=d_k(v)$, and dominated convergence uses the integrable bound $4\norm{d_k}^2\norm{\nabla f}_2^2$ after division by $h^2$.

We now extend this conditional DQM property to the entire trajectory by adding one observation at a time. At parameter $s+h$, let $a_h$ be the square root of the density of the history and current input, and let $b_h$ be the square root of the conditional density of the next state. These densities are relative to the corresponding factors of $\mu_t^\pi$, so $a_hb_h$ is the square root of the density after the next observation. Write $a=a_0$, $b=b_0$, with $L^2$ derivatives $a',b'$ at $h=0$. Subtracting the first-order product expansion gives
\begin{align*}
 &a_hb_h-ab-h(a'b+ab')\\
 &\quad=(a_h-a-ha')b_h+h a'(b_h-b)\\
 &\qquad\quad+a(b_h-b-hb').
\end{align*}
The first term is $o(|h|)$ in $L^2$ by induction and conditional normalization of $b_h$. The second is $o(|h|)$ because the conditional integral of $(b_h-b)^2$ tends to zero and is bounded by $4$, while $a'$ is square integrable. The last term is $o(|h|)$ by the preceding bound under \(\Pp_s\). Thus the product has the required first-order expansion. Appending the policy kernel leaves the history remainder unchanged because the kernel is independent of $s$. Starting from fixed $x_0$, induction proves trajectory DQM, with score $S_t(s)$ obtained by summing the conditional scores.

Here $w_{k+1}=x_{k+1}-A_sx_k-B_su_k$ under $\Pp_s$. Since $\E\psi(w)=0$, the summands of $S_t(s)$ are martingale differences. Their cross terms vanish, and conditional second moments give \eqref{eq:direction-FI}.
\end{proof}

We use the following form of the van Trees inequality.
On an interval, \(H^1\) denotes the square-integrable functions with square-integrable weak derivative, and \(H_0^1\) consists of those functions whose endpoint values vanish.
\begin{lemma}\label{lem:van-Trees}
Let \(\{P_s:s\in S\}\) be probability laws indexed by a real interval \(S\), with densities \(p_s\) relative to a common measure. Assume that \(s\mapsto p_s\) is weakly absolutely continuous with weak derivative \(p_sS_s\), where the score \(S_s\) is square integrable and has mean zero under \(P_s\). Let \(s\) have a density \(\lambda\) on a compact subinterval with \(\sqrt\lambda\in H_0^1\), let \(Y\sim P_s\) have Fisher information \(I_Y(s)=\E_sS_s^2\), and let \(g:S\to\R\) be absolutely continuous. Here \(\E_s\) denotes expectation under \(P_s\), and unindexed expectations use the joint prior and observation law. If the prior-averaged information is finite, every estimator \(\widehat g(Y)\) satisfies
\begin{equation}
 \E(\widehat g-g(s))^2
 \ge \frac{(\E g'(s))^2}{\E I_Y(s)+J(\lambda)}.       \label{eq:van-Trees}
\end{equation}
\end{lemma}

\begin{proof}
It suffices to consider estimators with finite mean squared error. Weak integration by parts under the joint law \(\lambda(s)P_s(dY)\), with zero boundary trace of \(\lambda\), gives
\begin{equation*}
 \E\!\left[(\widehat g-g(s))
 (S_s(Y)+\lambda'(s)/\lambda(s))\right]
 =\E g'(s).
\end{equation*}
The likelihood score has conditional mean zero and is orthogonal to the prior score. The second moment of their sum is therefore \(\E I_Y(s)+J(\lambda)\), so Cauchy--Schwarz proves \eqref{eq:van-Trees} (see \cite{vantrees2001,gill1995applications}).
\end{proof}

Choose once and for all a smooth density \(\lambda_0\) supported in \((-1,1)\) and satisfying \(J(\lambda_0)<\infty\). On the local line \(|s|\le aT^{-1/4}\), let
\begin{equation}
 \begin{aligned}
 \lambda_T(s)&=\frac{T^{1/4}}{a}\lambda_0\!\left(\frac{T^{1/4}s}{a}\right),\\
 J(\lambda_T)&=a^{-2}\sqrt T\,J(\lambda_0)=\Theta(\sqrt T).
 \end{aligned}                                             \label{eq:scaled-prior}
\end{equation}
This is the source of the $\sqrt T$ prior information term in both local lower bounds.

For each prior used in the lower bounds, assume first that the Bayes regret \(B_T(\lambda)\) is finite. Integrating \eqref{eq:generic-info} or \eqref{eq:null-info} over the prior gives finite prior-averaged Fisher information. The square-integrable weak gradient of \(f\) gives weak absolute continuity in \(s\) for \eqref{eq:app-dominated-density}, with weak derivative \(\ell_sS_t(s)\). The square integrability under the joint prior and observation law established in \eqref{eq:direction-FI} then justifies Fubini's theorem and weak integration by parts by Cauchy--Schwarz.

Take \(Y=(x_0,u_0,\ldots,u_{b-1},x_b)\) for the scalar function and projection in Section~\ref{sec:lower-reduction}. The projected estimator \(\widehat K_I\) is \(Y\)-measurable, so its scalar projection is admissible in \cref{lem:van-Trees}. With \(I_Y(s)=I_b^\pi(s;v)\), this verifies the application yielding \eqref{eq:controller-vt}.

\subsection{Early residual and energy bounds}\label{app:prefix-energy}
Use the local submodel \(\theta_s\), fixed policy \(\pi\), and conventions \(R_T(s),K_s,\E_s\) of Section~\ref{sec:lower-reduction}.
\begin{lemma}\label{lem:prefix-energy}
Under \cref{ass:noise-moments,ass:regular}, there is a constant \(C_E<\infty\), independent of the policy, such that, for every integer \(0\le t\le T/2\), every sufficiently large \(T\), and every local \(s\),
\begin{align}
 \E_s\sum_{k=0}^{t-1}\norm{u_k+K_sx_k}^2
 &\le C_E(R_T(s)+1),                                         \label{eq:residual-prefix}\\
 \E_s\sum_{k=0}^{t-1}(\norm{x_k}^2+\norm{u_k}^2)
 &\le C_E(t+R_T(s)+1).                                       \label{eq:energy-prefix}
\end{align}
\end{lemma}

\begin{proof}
Put \(e_k=u_k+K_{\theta_s,T-k}x_k\) for \(0\le k<t\le T/2\). Since \(S_{\theta_s,T-k}\succeq R\succ0\), \eqref{eq:FH-PDI} gives \(\E_s\sum_{k<t}\norm{e_k}^2\le C R_T(s)\).
By \cref{lem:benchmark} and Schur stability of \(F_{\theta_0}\), for all sufficiently large \(T\) the matrices \(A_{\theta_s}-B_{\theta_s}K_{\theta_s,T-k}\), \(k<T/2\), lie in a common stable neighborhood admitting a quadratic Lyapunov function. The dynamics have these matrices and forcing \(B_{\theta_s}e_k+w_{k+1}\). Conditional Lyapunov drift, using Young's inequality for the residual cross term and zero mean of the fresh noise, therefore yields
\begin{equation}
 \E_s\sum_{k<t}\norm{x_k}^2
 \le C\bigl(\norm{x_0}^2+t+R_T(s)\bigr).                   \label{eq:prefix-state-energy}
\end{equation}
Uniform boundedness of the finite-horizon gains and \(u_k=e_k-K_{\theta_s,T-k}x_k\) prove \eqref{eq:energy-prefix}. For the stationary residual, \eqref{eq:riccati-exp} gives
\begin{align*}
 \E_s\sum_{k<t}\norm{u_k+K_sx_k}^2
 &\le C R_T(s)+C\beta^T\E_s\sum_{k<t}\norm{x_k}^2\\
 &\le C\bigl(R_T(s)+1\bigr),
\end{align*}
where the last step uses \eqref{eq:prefix-state-energy}, fixed \(x_0\), and boundedness of \(T\beta^T\). This proves \eqref{eq:residual-prefix} with constants and a horizon threshold independent of \(s,t\) and the policy.
\end{proof}

\subsection{Bounds for block Gram matrices and gain estimation}\label{app:blockgram}
Use the local scalar submodel, policy, and expectation conventions of Section~\ref{sec:lower-reduction}. For a block of integer times $I=[\ell,b)\subset[1,T/2]$, write $n=b-\ell$. Let $V_I$ and the projected estimator $\widehat K_I$ be as defined at \eqref{eq:block-LS}.
\begin{lemma}\label{lem:block-risk}
Under \cref{ass:noise-moments,ass:regular}, there are constants \(n_0,c_B,C_B>0\), independent of \(\pi,\lambda,I,T\), such that, for all sufficiently large \(T\), if \(b\le2n\), \(n\ge n_0\), and \(\E_{s\sim\lambda}R_T(s)\le n\) for a prior supported on the local line, then
\begin{align}
 &\E_{s\sim\lambda}\E_s\sum_{t\in I}\norm{u_t+K_sx_t}^2 \notag\\
 &\qquad\ge c_B n\,\E_{s\sim\lambda}\E_s\fnorm{\widehat K_I-K_s}^2
 -\frac{C_B}{n}.                                             \label{eq:block-risk}
\end{align}
\end{lemma}

\begin{proof}
For every policy and parameter, \cref{lem:location-smallball} gives the conditional probability bound in \cref{lem:app-Gram}. The bound also holds under the joint law of the parameter and trajectory by conditioning first on the parameter. From \cref{lem:prefix-energy}, averaged over the prior,
\begin{equation*}
 \E_{s\sim\lambda}\E_s\tr(V_I)\le C(n+B_T(\lambda)+1)\le C'n.
\end{equation*}
Applying \cref{lem:app-Gram} under this joint law gives
\begin{equation}
 \E_{s\sim\lambda}\Pp_s(V_I\nsucceq c_gnI_{n_x})\le C_gn^{-2}.  \label{eq:app-blockGram}
\end{equation}

Set $e_t(s)=u_t+K_sx_t$ and $M_I=\sum_{t\in I}e_t(s)x_t^\top$. On the event $V_I\succeq c_gnI_{n_x}$, completing the square gives $M_IV_I^{-1}M_I^\top\preceq\sum_{t\in I}e_t(s)e_t(s)^\top$. The estimator in \eqref{eq:block-LS} therefore satisfies
\begin{align}
 \widetilde K_I-K_s
 &=-M_IV_I^{-1},\notag\\
 \fnorm{\widetilde K_I-K_s}^2
 &\le\frac{1}{c_gn}\tr\!\left(M_IV_I^{-1}M_I^\top\right)\notag\\
 &\le\frac{1}{c_gn}\sum_{t\in I}\norm{e_t(s)}^2,          \label{eq:app-Krisk}
\end{align}
where the first inequality uses $V_I^{-2}\preceq(c_gn)^{-1}V_I^{-1}$.
Projection onto $\cK$ cannot increase the distance to $K_s$. On the bad event, the projected squared error is bounded by $\operatorname{diam}(\cK)^2$. Averaging \eqref{eq:app-Krisk} and using \eqref{eq:app-blockGram} proves \eqref{eq:block-risk}.
\end{proof}

\section{Regret Bounds for a Known Uncertainty Radius}\label{app:radius}
\begin{proof}[Proof of \cref{thm:radius}]
Choose \(r_0\le1\) small enough for the local bounds in \cref{lem:benchmark,thm:cross-fiber,lem:left-inverse,lem:common-lyap} under their respective hypotheses. All constants may depend on \(x_0\), the fixed neighborhood and noise law, not on \(r,T\). The proof of \eqref{eq:nominal-upper-general} and smoothness give the nominal \(C(1+Tr^2)\) bound, or \(C(1+Tr^4)\) if \(DK_{\theta_0}=0\).

To prove \eqref{eq:radius-CI}, suppose CI holds. If \(\cM_{\theta_0}=\{0\}\), invariance of the optimal gain and nominal control give a uniform constant regret bound. Otherwise, use FCE with fixed projection radius \(r\) and epoch lengths
\begin{equation*}
 N_j=N_0(r)2^j,\qquad N_0(r)\triangleq\lceil4r^{-2}\rceil,
\end{equation*}
where \(N_0(r)\ge4\), and initialize \(\widehat\mu_0=0\). For \(\theta=\theta_0+\mu+\nu\in\Theta(r)\), the projection contains \(\mu\), and \(\norm{\mu},\norm{\nu},\norm{\widehat\mu_j}\le r\). The common Lyapunov bound gives uniform state second moments, so the proof for the projected OLS estimator applies at these deterministic epoch starts with constants independent of \(r,T\).

Equation \eqref{eq:nuisance-identity} gives \(\fnorm{\cZ_{K_j}\nu}\le Cr^2\). Applying the bounded left inverse, using nonexpansiveness of projection, and invoking \cref{thm:cross-fiber} gives
\begin{equation*}
 \E\norm{K_{j+1}-K_\theta}^2\le C/N_j+Cr^4,\qquad
 \norm{K_j-K_\theta}\le Cr.
\end{equation*}
The second bound is deterministic. Thus the second term on the right-hand side of \eqref{eq:epoch-loss} is at most \(Cr^2\). The initial epoch contributes at most \(C[1+\min\{T,N_0(r)\}r^2]\le C\). For every later epoch, including a final incomplete epoch, its length is at most \(2N_{j-1}\), so its statistical contribution is at most a constant. There are at most \(C[1+\log(1+Tr^2)]\) epochs. Summing their boundary terms and the residual \(Cr^4\) per step proves \eqref{eq:radius-CI}. If \(\cN_{\theta_0}=\mathcal D\), then \(A[\nu]=B[\nu]=0\) and \(K_{\theta_0+\mu+\nu}=K_{\theta_0+\mu}\). Both error terms due to the invisible component vanish, so the \(Tr^4\) term is absent.

For the lower bounds, additionally impose \cref{ass:location-noise}. The proof of \cref{lem:prefix-energy} extends \eqref{eq:residual-prefix}--\eqref{eq:energy-prefix} to \(|s|\le r_0\). On this fixed neighborhood, the gains \(K_s\) lie in a compact ball \(\cK\), while \(J_w\) and the constants in \cref{lem:location-smallball} do not depend on \(r,T\). The proofs of \eqref{eq:generic-info} and \cref{lem:block-risk} therefore apply with uniform constants.

Fix a unit \(v\) with \(G\triangleq DK_{\theta_0}[v]\ne0\), reducing \(r_0\) if necessary so that \(g(s)\triangleq\ip{G/\fnorm{G}}{K_{\theta_0+sv}}_F\) has \(g'(s)\ge\fnorm{G}/2\) on \([-r_0,r_0]\). For \(0<h\le r_0\), scale the smooth prior from Appendix~\ref{app:info} as \(\lambda_h(s)=h^{-1}\lambda_0(s/h)\). Its information is \(h^{-2}J(\lambda_0)\). Assuming \(B_T(\lambda_h)<\infty\), the energy bound gives finite averaged information and justifies the same weak integration by parts. Hence
\begin{equation}
 \E_{s\sim\lambda_h}\E_s\fnorm{\widehat K_I-K_s}^2
 \ge\frac{c}{\E_{s\sim\lambda_h}I_b^\pi(s;v)+Ch^{-2}}.          \label{eq:radius-vt}
\end{equation}

For the logarithmic lower bound \eqref{eq:radius-log}, take \(h=r\) and set \(X\triangleq Tr^2\). If \(B_T(\lambda_h)\ge\sqrt X\), then \eqref{eq:radius-log} already follows. Otherwise, choose
\begin{equation*}
 m\triangleq\lceil M\max\{r^{-2},\sqrt X\}\rceil,
 \qquad I_j=[2^jm,2^{j+1}m),
\end{equation*}
with \(j=0,1,\ldots\), retaining blocks with \(2^{j+1}m\le T/2\). Choose the fixed \(M\) large enough for the block lemma and its remainder. On every retained block of length \(n\), both \(B_T(\lambda_h)\) and \(r^{-2}\) are at most \(n\). The denominator in \eqref{eq:radius-vt} is therefore at most \(Cn\), and \eqref{eq:block-risk} gives a positive constant residual loss per block. For \(X\ge1\), since \(r\le1\),
\begin{equation*}
 \frac{T}{m}\ge c\min\{X,T/\sqrt X\}\ge c\sqrt X.
\end{equation*}
Consequently there are at least \(c\log(1+X)-C\) retained blocks. Summing their losses and applying \eqref{eq:residual-prefix} gives \(B_T(\lambda_h)\ge c\log(1+X)-C\). Bounded \(X\) is covered by the additive constant and nonnegativity of finite-horizon regret. Taking the supremum over the prior support proves \eqref{eq:radius-log}.

To prove \eqref{eq:radius-hard} when CI fails, take a unit \(v\in\ker\cZ_{\theta_0}\) with \(DK_{\theta_0}[v]\ne0\), and let \(h=\min\{r,T^{-1/4}\}\). The exact score split \eqref{eq:null-score-split}, smoothness, and the energy bounds on the fixed neighborhood give, for \(t\le T/2\),
\begin{equation*}
 I_t^\pi(s;v)\le C\bigl[R_T(s)+Th^2+1\bigr],\quad |s|\le h.
\end{equation*}
Use \(I=[n,2n)\) with \(n=\lfloor T/4\rfloor\). If \(B_T(\lambda_h)\ge n\), the desired bound is immediate for large \(T\). Otherwise combine \eqref{eq:radius-vt}, \eqref{eq:block-risk}, and \eqref{eq:residual-prefix}. Since \(Th^2\le h^{-2}\), they yield
\begin{equation}
 B_T(\lambda_h)+C\ge\frac{cT}{B_T(\lambda_h)+h^{-2}+C}.               \label{eq:radius-self-bound}
\end{equation}
Put \(q\triangleq Th^2=\min\{Tr^2,\sqrt T\}\). If \(B_T(\lambda_h)+C\ge q\), the claim follows. Otherwise \(B_T(\lambda_h)+C<q\le h^{-2}\), so \eqref{eq:radius-self-bound} gives \(B_T(\lambda_h)+C\ge cTh^2=cq\). This proves \eqref{eq:radius-hard}. All prior supports lie inside \(\Theta(r)\), and the argument applies to every causal, possibly randomized policy.
\end{proof}

\begin{proof}[Proof of \cref{cor:fixed-radius}]
Corollary~\ref{cor:physical-identifiability} gives \(\cN_{\theta_0}=\mathcal D\), so \cref{thm:radius} removes the \(Tr^4\) term. For fixed \(r>0\), its upper bound is logarithmic in \(T\), and \eqref{eq:radius-log} gives the matching lower bound under the additional conditions.
\end{proof}

\bibliographystyle{unsrt}
\bibliography{references}

\end{document}